\documentclass[final,reqno]{elsarticle}
\usepackage{mathrsfs}
\usepackage{rotating} 

\usepackage{amssymb,amsmath,bbm,amsfonts,latexsym,subfigure}
\usepackage{amsthm}
\usepackage{graphicx,float,epsfig,color,fancyhdr}

\newtheorem{lemma}{Lemma}[section]
\newtheorem{remark}{Remark}[section]

\newtheorem{theorem}{Theorem}[section]
\newtheorem{corollary}{Corollary}[section]
\newtheorem{problem}{Problem}

\def\CE{\mathcal{E}}

\def\CT{\mathcal{T}}
\def\E{K}
\def\G{\Gamma}
\def\Go{\G_0}

\def\HuE{H^1(\E)}

\def\HuO{H^1(\O)}

\def\LO{L^2(\O)}
\def\N{\mathbb{N}}
\def\O{\Omega}
\def\P{\mathbb{P}}
\def\R{\mathbb{R}}
\def\dim{\mathop{\mathrm{\,dim}}\nolimits}
\def\disp{\displaystyle}

\def\g{\gamma}

\def\hCT{\widehat{\CT}}
\def\hg{\widehat{\g}}

\def\l{\lambda}
\def\PiK{\Pi_k^{\E}}

\def\Vh{V_h}
\def\VK{V^{\E}_k}

\allowdisplaybreaks
\journal{}
\date{\today}
\begin{document}
\begin{frontmatter}

\title{A posteriori error estimates for a Virtual Elements Method for the Steklov eigenvalue problem.}

\author[1,2]{David Mora}
\ead{dmora@ubiobio.cl}
\address[1]{Departamento de Matem\'atica,
Universidad del B\'io-B\'io, Casilla 5-C, Concepci\'on, Chile.}
\author[2,3]{Gonzalo Rivera}
\ead{grivera@ing-mat.udec.cl}
\address[2]{Centro de Investigaci\'on en Ingenier\'ia Matem\'atica
(CI$^2$MA), Universidad de Concepci\'on, Casilla 160-C, Concepci\'on, Chile.}
\author[2,3]{Rodolfo Rodr\'iguez}
\ead{rodolfo@ing-mat.udec.cl}
\address[3]{Departamento de Ingenier\'ia Matem\'atica,
Universidad de Concepci\'on, Casilla 160-C, Concepci\'on, Chile.}

\begin{abstract} 
The paper deals with the a posteriori error analysis of a virtual element method for the Steklov
eigenvalue problem.
The virtual element method 
has the advantage of using general polygonal meshes,
which allows implementing very efficiently mesh refinement strategies.
We introduce a residual type a posteriori error estimator and prove its reliability and efficiency. We use the corresponding error estimator to drive an adaptive scheme. Finally,
we report the results of a couple of numerical tests, that allow us to assess the performance of this approach.

\end{abstract}


\end{frontmatter}

%-----------------------------------------------------------------------

\setcounter{equation}{0}
\section{Introduction}
\label{SEC:INTR}

The {\it Virtual Element Method} (VEM), introduced in \cite{BBCMMR2013,BBMR2014},
is a recent generalization of the Finite Element Method,
which is characterized by the capability of dealing with
very general polygonal/polyhedral meshes.
The interest in numerical methods that can make use of general
polytopal meshes has recently undergone a significant growth
in the mathematical and engineering literature; among
the large number of papers on this subject, we cite as a minimal sample
\cite{BBCMMR2013,BLMbook2014,CGH14,DPECMAME2015,RW,ST04,TPPM10}.
Indeed, polytopal meshes can be very useful for
a wide range of reasons including meshing of
the domain,
automatic use of hanging nodes,  moving meshes and adaptivity.
VEM has been applied successfully in a
large range of problems; see for instance
\cite{AABMR13,ABMVsinum14,ALM15,BBCMMR2013,BBMR2014,BLM2015,BMR2016,BMRR,BBPS2014,CG2016,Paulino-VEM,MRR2015,PG2015,PPR15}.

The object of this paper is to introduce and analyze an a posteriori
error estimator of residual type for the virtual element approximation
of the Steklov eigenvalue problem.
In fact, due to the large flexibility of the meshes to which
the virtual element method is applied, mesh adaptivity becomes
an appealing feature as mesh refinement strategies can
be implemented very efficiently. For instance, hanging nodes can be introduced
in the mesh to guarantee the mesh conformity without spreading the refined
zones. In fact hanging nodes introduced by the refinement of a
neighboring element are simply treated as new nodes since adjacent
non matching element interfaces are perfectly acceptable. On the other hand, polygonal cells
with very general shapes are admissible thus allowing us to adopt
simple mesh coarsening algorithms.

The approximation of eigenvalue problems has been the object of
great interest from both the practical and theoretical points of view,
since they appear in many  applications.
We refer to \cite{Boffi} and the references therein for the
state of art in this subject area. In particular, 
the  Steklov eigenvalue problem, which involves the Laplace
operator but is characterized by the presence of the eigenvalue
in the boundary condition, appears in many applications;
for example, the study of the vibration modes of a structure
in contact with an incompressible fluid (see \cite{BRS00})
and the analysis of the stability of mechanical oscillators
immersed in a viscous media (see \cite{PT}). One of its main
applications arises from the dynamics of liquids in moving containers,
i.e., sloshing problems (see \cite{BRS03,CM79,Ch93,CY96,EM87,WE91}).

On the other hand, adaptive mesh refinement strategies based
on a posteriori error indicators play a relevant role in the
numerical solution of partial differential equations in a general
sense. For instance, they guarantee achieving errors below a tolerance with a reasonable computer cost in presence of singular solutions. Several approaches have been
considered to construct error estimators based on the residual
equations (see \cite{AMOJT,DPR2003,Verfurth} and the references therein).
In particular, for the Steklov eigenvalue problem we mention
\cite{A_M2AN2004,AP_APNUM2009,DrAA2011,GM-ima2011,XIMA13}. On the other
hand, the design and analysis  of a posteriori error bounds for the VEM
is a challenging task. References \cite{BMm2as,CGPS} are the only a
posteriori error analyses   for VEM currently available in the
literature. In \cite{BMm2as}, a posteriori error bounds for the
$C^{1}$-conforming VEM for the two-dimensional Poisson problem are
proposed.  In \cite{CGPS}, a posteriori error bounds are introduced for
the $C^{0}$-conforming VEM proposed in \cite{CMS2015} for the
discretization of second order linear elliptic
reaction-convection-diffusion problems with non constant coefficients in
two and three dimensions.

 We have recently developed in \cite{MRR2015} a virtual element method
for the Steklov eigenvalue problem. Under standard assumptions on the
computational domain, we have established that the resulting scheme provides
a correct approximation of the spectrum and proved optimal order
error estimates for the eigenfunctions and a double order for the eigenvalues. In order to
exploit the capability of VEM in the use of general polygonal meshes
and its flexibility for the application of mesh adaptive strategies,
we introduce and analyze an a posteriori error estimator for the
virtual element approximation introduced  in \cite{MRR2015}.
Since  normal fluxes of the VEM solution are not computable, they  will be replaced in the estimators 
by a proper projection. As a consequence of this replacement,
new additional terms appear in the a posteriori error estimator,
which represent the virtual inconsistency of VEM. Similar  terms also
appear in the other papers for a posteriori error estimates of
VEM (see \cite{BMm2as,CGPS}). We prove that the error estimator
is equivalent to the error and use the corresponding indicator to drive an adaptive scheme.

The outline of this article is as follows: in
Section~\ref{SEC:Apriori} we present the continuous and
discrete formulations of the Steklov eigenvalue problem
together with  the spectral characterization.
Then, we recall the a priori error estimates for the virtual element
approximation analyzed in \cite{MRR2015}. In Section~\ref{SEC:PosterioriError},
we define the a posteriori error estimator and proved its reliability and efficiency.
Finally, in Section~\ref{SEC:NUMERejemplo}, we report a set of numerical
tests that allow us to assess the performance of  an adaptive strategy driven by the estimator. We have also made a comparison between
the proposed estimator and the standard edge-residual error estimator for a finite element method.

Throughout the article we will denote by $C$ a generic constant
independent of the mesh parameter $h$, which may take different values
in different occurrences.

%-----------------------------------------------------------------------

\setcounter{equation}{0}
\section{The Steklov eigenvalue problem and its virtual element approximation}
\label{SEC:Apriori}
Let $\O\subset\R^2$ be a bounded domain with polygonal
boundary $\partial\O$. Let $\Go$ and $\G_1$ be disjoint
open subsets of $\partial\O$ such that $\partial\O=\bar{\G}_0\cup\bar{\G}_1$
with $\Go\ne\emptyset$. We denote by $n$ the outward
unit normal vector to $\partial\O$.

We consider the following eigenvalue problem:

Find $(\l,w)\in\R\times\HuO$, $w\ne0$, such that
$$
\left\{\begin{array}{l}
\Delta w=0\quad\text{in }\O,
\\[0.1cm]
\dfrac{\partial w}{\partial n}
=\left\{\begin{array}{ll}
\l w & \text{on }\Go,
\\
0 & \text{on }\G_1.
\end{array}\right.
\end{array}\right. 
$$

By testing the first equation above with $v\in\HuO$
and integrating by parts, we arrive at the following
equivalent weak formulation:
\begin{problem}
\label{P1}
Find $(\l,w)\in\R\times\HuO$, $w\ne0$, such that
$$
\int_{\O}\nabla w\cdot\nabla v
=\l\int_{\Go}wv\qquad\forall v\in\HuO.
$$
\end{problem}

According to \cite[Theorem 2.1]{MRR2015}, we know that the solutions
$(\l,w)$ of the problem above are:
\begin{itemize}
\item$\l_{0}=0$, whose associated eigenspace is the space of
constant functions in $\O$;
\item a sequence of positive finite-multiplicity
eigenvalues $\{\l_{k}\}_{k\in \N}$ such that $\lambda_k\rightarrow\infty$. 
\end{itemize}
The eigenfunctions
corresponding to different eigenvalues are orthogonal in $L^{2}(\G_{0})$.
Therefore the eigenfunctions $w^{k}$ corresponding to $\l_{k}>0$ satisfy 
\begin{equation}
\label{rr1}
\int_{\Go}w^{k}=0.
\end{equation}

We denote the bounded bilinear symmetric forms appearing in Problem \ref{P1} as follows:
\begin{align*}
%\label{ab}
a(w,v)
& :=\int_{\O}\nabla w\cdot\nabla v,
\qquad w,v\in\HuO,
\\
%\label{ba}
b(w,v)
& :=\int_{\Go}wv,
\qquad w,v\in\HuO.
\end{align*}

Let $\left\{\CT_h\right\}_h$ be a sequence of decompositions of $\O$
into polygons $\E$. We assume that for every mesh $\CT_h$, $\overline{\G}_{0}$ and $\overline{\G}_{1}$ are union of edges of elements $\E\in \CT_{h}$. Let $h_\E$ denote the diameter of the element
$\E$ and $h$ the maximum of the diameters of all the elements of the
mesh, i.e., $h:=\max_{\E\in\CT_h}h_\E$.

For the analysis, we will make as in \cite{BBCMMR2013,MRR2015} the following assumptions.
\begin{itemize}
\item \textbf{A1.} Every mesh $\CT_h$ consists of a finite number of 
\textit{simple} polygons ({i.e.}, open simply connected sets with
non self intersecting polygonal boundaries).
\item \textbf{A2.} There exists $\g>0$ such that, for all meshes
$\CT_h$, each polygon $\E\in\CT_h$ is star-shaped with respect to a ball
of radius greater than or equal to $\g h_{\E}$.
\item \textbf{A3.} There exists $\hg>0$ such that, for all meshes
$\CT_h$, for each polygon $\E\in\CT_h$, the distance between any two of
its vertices is greater than or equal to $\hg h_{\E}$.
\end{itemize}

We consider now a simple polygon $\E$ and, for $k\in\N$, we define 
$$
\mathbb{B}_k(\partial\E)
:=\left\{v\in C^0(\partial\E): \ v|_{\ell}\in\P_k(\ell)
\text{ for all edges }\ell\subset\partial\E\right\}.
$$
We then consider the finite-dimensional space defined as follows:
\begin{equation}\label{Vk}
\VK
:=\left\{v\in\HuE:
\ v|_{\partial\E}\in\mathbb{B}_k(\partial\E)
\text{ and }\Delta v|_{\E}\in\P_{k-2}(\E)\right\},
\end{equation}
where, for $k=1$, we have used the convention that
$\P_{-1}(\E):=\left\{0\right\}$. We choose in this
space the degrees of freedom introduced in \cite[Section~4.1]{BBCMMR2013}.
Finally, for every decomposition $\CT_h$ of $\O$ into simple
polygons $\E$ and for a fixed $k\in\N$, we define
$$
\Vh:=\left\{v\in\HuO:\ v|_{\E}\in\VK\quad\forall\E\in\CT_h\right\}.
$$
In what follows, we will use standard Sobolev spaces, norms and seminorms and also  the broken $H^1$-seminorm
$$
\left|v\right|_{1,h}^2
:=\sum_{\E\in\CT_h}\left\|\nabla v\right\|_{0,\E}^2,
$$
which is well defined for every $v\in\LO$ such that $v|_{\E}\in\HuE$ for each polygon $\E\in\CT_h$.

We split the bilinear form $a(\cdot,\cdot)$ as follows:
$$
a(u,v)=\sum_{\E\in\CT_h}a^{\E}(u,v),
\qquad u,v\in\HuO,
$$
where 
\begin{equation*}
%\label{defin}
a^{\E}(u,v)
:=\int_{\E}\nabla u\cdot\nabla v,
\qquad u,v\in\HuE.
 \end{equation*}
Due to the implicit space definition,  we must have into account
that  we would not know how
to compute $a^{\E}(\cdot,\cdot)$ for $u_h,v_h\in V_h$. Nevertheless, the final
output will be a local matrix on each element $\E$ whose associated
bilinear form can be  exactly computed whenever one of the two entries is a polynomial
of degree $k$. This will allow us to retain the optimal approximation
properties of the space $V_h$.

With this end, for any $\E\in\CT_h$ and for any sufficiently
regular function $\varphi$, we define first
\begin{equation*}%\label{equd}
\overline{\varphi}:=\frac1{N_{\E}}\sum_{i=1}^{N_{\E}}\varphi(P_i),
\end{equation*}
where $P_i$, $1\le i\le N_{\E}$, are the vertices of $\E$. Then, we
define the projector $\PiK:\ \VK\longrightarrow\P_k(\E)\subseteq\VK$ for
each $v_h\in\VK$ as the solution of 
\begin{subequations}
%\label{23}
\begin{align*}
a^{\E}\big(\PiK v_h,q\big)
& =a^{\E}(v_h,q)
\qquad\forall q\in\P_k(\E),
%\label{numero}
\\
\overline{\PiK v_h}
&=\overline{v_h}.
%\label{numeroo}
\end{align*}
\end{subequations}

On the other hand, let $S^{\E}(\cdot,\cdot)$ be any symmetric
positive definite bilinear form to be chosen as to satisfy 
\begin{equation}
\label{20}
c_0\,a^{\E}(v_h,v_h)\leq S^{\E}(v_h,v_h)\leq c_1\,a^{\E}(v_h,v_h)
\qquad\forall v_h\in\VK\text{ with }\PiK v_h=0
\end{equation}
for some positive constants $c_0$ and $c_1$ independent of $\E$. Then, set
$$
a_h(u_h,v_h)
:=\sum_{\E\in\CT_h}a_h^{\E}(u_h,v_h),
\qquad u_h,v_h\in\Vh,
$$
where $a_h^{\E}(\cdot,\cdot)$ is the bilinear form defined on
$\VK\times\VK$ by
\begin{equation*}
%\label{21}
a_h^{\E}(u_h,v_h)
:=a^{\E}\big(\PiK u_h,\PiK v_h\big)
+S^{\E}\big(u_h-\PiK u_h,v_h-\PiK v_h\big),
\qquad u_h,v_h\in\VK. 
\end{equation*}
Notice that the bilinear form $S^{\E}(\cdot,\cdot)$ has to be actually
computable for $u_h,v_h\in V_{k}^{\E}$.

The following properties of $a_h^{\E}(\cdot,\cdot)$
have been established in \cite[Theorem~4.1]{BBCMMR2013}.
\begin{itemize}
\item \textit{$k$-Consistency}: 
\begin{equation}
\label{consistencia}
a_h^{\E}(p,v_h)
=a^{\E}(p,v_h)
\qquad\forall p\in\P_k(\E),
\quad\forall v_h\in\VK. 
\end{equation}
\item \textit{Stability}: There exist two positive constants $\alpha_*$
and $\alpha^*$, independent of $\E$, such that:
\begin{equation}
\label{24}
\alpha_*a^{\E}(v_h,v_h)
\leq a_h^{\E}(v_h,v_h)
\leq\alpha^*a^{\E}(v_h,v_h)
\qquad\forall v_h\in\VK. 
\end{equation}
\end{itemize}
 
Now, we are in a position to write the virtual
element discretization of Problem~\ref{P1}.
\begin{problem}
\label{P11}
Find $(\l_h,w_h)\in\R\times\Vh$, $w_h\ne0$, such that
$$
a_h(w_h,v_h)
=\l_h b(w_h,v_h)
\qquad\forall v_h\in\Vh.
$$
\end{problem}

According to \cite[Theorem 3.1]{MRR2015} we know that the solutions
$(\l_{h},w_{h})$ of the problem above are:
\begin{itemize}
\item$\l_{h0}=0$, whose associated eigenfunction are the constant functions in $\O$.
\item$\{\l_{hk}\}_{k=1}^{N_{h}}$, with $N_{h}:=\dim\left\{v_{h}|_{\G_{0}}\text{, }v_{h}\in V_{h}\right\}-1$, which are  positive eigenvalues
repeated according to their respective multiplicities.
\end{itemize}
Moreover, the eigenfunctions corresponding to different eigenvalues are orthogonal
in $L^{2}(\G_{0})$. Therefore the eigenfunctions $w_{h}^{k}$ corresponding to $\l_{hk}>0$ satisfy 
\begin{equation}
\label{rr2}
\int_{\Go}w_{h}^{k}=0.
\end{equation}

Let $(\l,w)$ be a solution to Problem~\ref{P1}.
We assume $\l>0$ is a simple eigenvalue and we normalize $w$ so that
$\left\|w\right\|_{0,\G_{0}}=1$. Then, for each mesh $\CT_{h}$, there
exists a solution $(\l_{h},w_{h})$ of Problem~\ref{P11}
such that $\l_h\rightarrow\l$,  $\left\|w_{h}\right\|_{0,\G_{0}}=1$ and $\|w-w_{h}\|_{1,\O}\rightarrow 0$
as $h\rightarrow 0$. Moreover, according to \eqref{rr1} and \eqref{rr2},
we have that $w$ and $w_{h}$ belong to the space 
\begin{equation*}
%\label{errore}
V:= \left\{v\in \HuO:\int_{\Go}v=0\right\}. 
\end{equation*}
Let us remark that the following generalized Poincar\'e inequality holds true
in this space: there exists $C>0$ such that 
\begin{equation}
\label{poincare1}
\|v\|_{1,\O}\leq C|v|_{1,\O}\qquad \forall v\in V.
\end{equation}

The following a priori error estimates have been proved
in \cite[Theorems 4.2--4.4]{MRR2015}: 
there exists $C>0$ such that for all $r\in[\frac{1}{2},r_{\O})$
\begin{align}\label{eq29}
\left\|w-w_h\right\|_{1,\O}&\le Ch^{\min\{r,k\}},\\\label{eq292}
\left|\l-\l_h\right|&\le Ch^{2\min\{r,k\}},\\
\left\|w-w_h\right\|_{0,\G_0}&\le Ch^{\min\{r,1\}/2+\min\{r,k\}},\label{eq293}
\end{align}
where the constant $r_\O> \frac{1}{2}$ is the Sobolev exponent
for the Laplace problem with Neumann boundary conditions. Let us remark that $r_\O>1$, if $\O$ is convex, and $r_\O:=\frac{\pi}{\omega}$ with $\omega$ being the largest re-entrant angle of $\O$, otherwise.

% --------------------------------------
\setcounter{equation}{0}
\section{A posteriori error analysis}
\label{SEC:PosterioriError}
The aim of this section is to introduce a suitable
residual-based error estimator for the Steklov
eigenvalue problem which be fully computable,
in the sense that  it depends only on quantities available from the VEM solution. Then, we will show its equivalence with the error.
For this purpose, we introduce the following definitions and notations.

For any polygon $\E\in \CT_h$, we denote by $\CE_{\E}$ the set of edges of $\E$
and 
$$\CE:=\bigcup_{\E\in\CT_h}\CE_{\E}.$$
We decompose $\CE=\CE_{\O}\cup\CE_{\Go}\cup\CE_{\G_1}$,
where $\CE_{\Go}:=\{\ell\in \CE:\ell\subset \Go\}$, $\CE_{\G_1}:=\{\ell\in \CE:\ell\subset \G_1\}$
and $\CE_{\O}:=\CE\backslash(\CE_{\Go}\cup\CE_{\G_1})$.
For each inner edge $\ell\in \CE_{\O}$ and for any  sufficiently smooth  function
$v$, we define the jump of its normal derivative on $\ell$ by
$$\left[\!\!\left[ \dfrac{\partial v}{\partial{ n}}\right]\!\!\right]_\ell:=\nabla (v|_{\E})  \cdot n_{\E}+\nabla ( v|_{\E'}) \cdot n_{\E'} ,$$
where $\E$ and $\E'$ are  the two elements in $\CT_{h}$  sharing the
edge $\ell$ and $n_{\E}$ and $n_{\E'}$ are the respective outer unit normal vectors.
 
As a consequence of the mesh regularity assumptions,
we have that each polygon $\E\in\CT_h$ admits a sub-triangulation $\CT_h^{\E}$
obtained by joining each vertex of $\E$ with the midpoint of the ball with respect
to which $\E$ is starred. Let $\hCT_h:=\bigcup_{\E\in\CT_h}\CT_h^{\E}$.
Since we are also assuming \textbf{A3}, $\big\{\hCT_h\big\}_h$
is a shape-regular family of triangulations of $\O$. 

We introduce bubble functions on polygons as follows (see  \cite{CGPS}).
 An interior bubble function $\psi_{\E}\in H_0^1(\E)$ for a polygon $\E$
can be constructed piecewise as the sum of the  cubic
bubble functions  for each triangle of the
sub-triangulation $\CT_h^{\E}$ that attain the value 1 at the barycenter of each triangle. On the other hand, an edge bubble function $\psi_{\ell}$ for
$\ell\in\partial \E$ is a piecewise quadratic function
attaining the value 1 at the barycenter of $\ell$ and vanishing
on the triangles $T\in\hCT_h$  that do not contain $\ell$
on its boundary.

The following results which establish standard estimates
for bubble functions will be useful in what follows (see \cite{AMOJT,Verfurth}).

\begin{lemma}[Interior bubble functions]
\label{burbujainterior}
For any $\E\in \CT_h$, let $\psi_{\E}$ be the corresponding interior bubble function.
Then, there exists a constant $C>0$ 
independent of  $h_\E$ such that
\begin{align*}
C^{-1}\|q\|_{0,\E}^2&\leq \int_{\E}\psi_{\E} q^2\leq \|q\|_{0,\E}^2\qquad \forall q\in \P_k(\E),\\
C^{-1}\| q\|_{0,\E}&\leq \|\psi_{\E} q\|_{0,\E}+h_\E\|\nabla(\psi_{\E} q)\|_{0,\E}\leq C\|q\|_{0,\E}\qquad \forall q\in \P_k(\E).
\end{align*}
\end{lemma}
\begin{lemma}[Edge bubble functions]
\label{burbuja}
For any $\E\in \CT_h$ and $\ell\in\CE_{\E}$, let $\psi_{\ell}$
be the corresponding edge bubble function. Then, there exists
a constant $C>0$ independent of $h_\E$ such that
 \begin{equation*}
C^{-1}\|q\|_{0,\ell}^2\leq \int_{\ell}\psi_{\ell} q^2 \leq \|q\|_{0,\ell}^2\qquad
\forall q\in \P_k(\ell).
\end{equation*}
Moreover, for all $q\in\P_k(\ell)$, there exists an extension of  $q\in\P_k(\E)$ (again denoted by $q$) such that
 \begin{align*}
h_\E^{-1/2}\|\psi_{\ell} q\|_{0,\E}+h_\E^{1/2}\|\nabla(\psi_{\ell} q)\|_{0,\E}&\leq C\|q\|_{0,\ell}.
\end{align*}
\end{lemma}
\begin{remark}
\label{extencion}
A possible way of extending $q$ from $\ell\in\CE_{\E}$ to $\E$
so that Lemma~\ref{burbuja} holds is  as follows:
first we extend $q$ to the straight line $L\supset\ell$ using the same polynomial function.
Then, we extend it to the whole plain through a constant
prolongation in the normal direction to $L$. Finally, we restrict  the latter to $\E$. 
\end{remark}

The following lemma provides an error equation which will be the starting
point of our error analysis. From now on, we will denote by
$e:=(w-w_{h})\in V$ the eigenfunction error and by 
\begin{equation}
\label{saltocal}
J_{\ell}:=\left\{\begin{array}{ll}
\dfrac{1}{2}\left[\!\!\left[ \dfrac{\partial (\PiK  w_{h})}{\partial{ n}}\right]\!\!\right]_{\ell},\quad & \ell\in \CE_{\O},
\\[0.4cm]
\l_h w_h-\dfrac{\partial (\PiK  w_{h})}{\partial n},\quad & \ell\in\CE_{\G_0},
\\[0.4cm]
-\dfrac{\partial (\PiK  w_{h})}{\partial n},\quad & \ell\in\CE_{\G_1},
\end{array}\right. 
\end{equation}
the edge residuals.  Notice that $J_{\ell}$ are actually computable since they only  involve values of $w_{h}$ on $\G_{0}$ (which are computable in terms of the boundary degrees of freedom) and $\PiK  w_{h}\in\P_{k}(\E)$ which is also computable.
\begin{lemma}
\label{ext2}
For any $v\in \HuO$, we have the following identity:
$$
a(e,v)=\l b(w,v)-\l_h b(w_h,v)-\sum_{\E\in \CT_h}a^\E(w_h-\PiK  w_h,v)
+\sum_{\E\in \CT_h}\left[\int_{\E}\Delta (\PiK  w_h)v
+\sum_{\ell\in \CE_{\E}}\int_{\ell}J_{\ell} v\right].
$$
\end{lemma}
\begin{proof} Using that $(\l,w)$ is a solution of Problem~\ref{P1},
adding and subtracting $\PiK w_h$ and integrating by parts, we obtain 
\begin{align*}
a(e,v)&=\l b(w,v)-a(w_h,v)
\\
&=\l b(w,v)
-\sum_{\E\in\CT_h}\left[a^\E(w_h-\PiK w_h,v)+a^\E(\PiK w_h,v)\right]
\\
&=\l b(w,v)-\sum_{\E\in \CT_h}a^\E(w_h-\PiK w_h,v)
-\sum_{\E\in \CT_h}\left[-\int_{\E}\Delta(\PiK w_h)\,v
+\int_{\partial \E}\dfrac{\partial(\PiK w_{h})}
{\partial{n}}\,v\right]
\\
&=\l b(w,v)-\sum_{\E\in\CT_h}a^\E(w_h-\PiK w_h,v)
\\
&\quad +\sum_{\E\in\CT_h}\left[\int_{\E}\Delta(\PiK w_h)\,v
-\!\!\!\!\!\!\sum_{\ell\in\CE_{\E}\cap(\CE_{\Go}\cup\CE_{\G_1})}
\int_{\ell}\dfrac{\partial(\PiK w_{h})}{\partial{n}}\,v
+\dfrac{1}{2}\sum_{\ell\in\CE_{\E}\cap\CE_{\O}}
\int_{\ell}\left[\!\!\left[\dfrac{\partial(\PiK w_{h})}
{\partial{n}}\right]\!\!\right]_{\ell}v\right].
\end{align*}
Finally, the proof follows by adding and subtracting the term $\l_h b(w_h,v)$.
\end{proof}
For all $\E\in \CT_{h}$, we introduce the  local terms
$\theta_{\E}$ and $R_{\E}$ and the local error indicator $\eta_{\E}$ by
\begin{align*}
%\label{RK}
\theta_{\E}^2&:=a_h^\E(w_h-\PiK  w_h,w_h-\PiK  w_h),\\
%\label{etaK1}
R_{\E}^{2}&:=h_{\E}^{2}\|\Delta (\PiK  w_h)\|_{0,\E}^{2},\\
%\label{etaK}
\eta_{\E}^{2}&:= \theta_{\E}^2+R_{\E}^{2}+\sum_{\ell \in \CE_{\E}}h_\E\|J_{\ell}\|_{0,\ell}^2.
\end{align*}
We also introduce  the global error estimator by
\begin{align*}
%\label{etacald}
\eta^{2}&:=\sum_{\E\in \CT_h}\eta_{\E}^2.     
\end{align*}

\begin{remark}
\label{RKK}
The indicators $\eta_{\E}$ include the terms $\theta_{\E}$ which do not appear in standard finite element estimators. This term, which represent the virtual
inconsistency of the method, has been  introduced
in \cite{BMm2as,CGPS} for a posteriori error estimates of other VEM. Let us emphasize that it can be directly computed in terms of the bilinear form $S^{\E}(\cdot,\cdot)$.
In fact, $$\theta_{\E}^2=a_h^\E(w_h-\PiK  w_h,w_h-\PiK  w_h)=S^\E(w_h-\PiK  w_h,w_h-\PiK  w_h).$$
\end{remark}

\subsection{Reliability of the a posteriori error estimator}

First, we  provide an upper bound for the error.

\begin{theorem}
\label{erroipo}
There exists a constant $C>0$ independent of $h$ such that
\begin{align*}
|w- w_h|_{1,\O} &\leq C\left(\eta+\dfrac{\l+\l_h}{2}\|w-w_{h}\|_{0,\G_0}\right).
\end{align*}
\end{theorem}
\begin{proof}
Since $e=w-w_h\in V\subset \HuO$, there exists $e_I\in \Vh$ satisfying (see \cite[Proposition 4.2]{MRR2015})
\begin{equation}\label{cotainter}
\|e-e_{I}\|_{0,\E}+h_{\E}|e-e_{I}|_{1,\E}\leq Ch_{\E}\|e\|_{1,\E}.
\end{equation} 
Then, we have that
\begin{equation}\label{cotasalto}
\begin{split}
|w- w_h|_{1,\O}^2 &=   a(w- w_h,e)\\
&=a(w-w_h,e-e_I)+a(w,e_I)-a_h(w_h,e_I)+a_h(w_h,e_I)-a( w_h,e_I)\\
&=\underbrace{\l b( w,e)-\l_h b(w_h,e)}_{T_{1}}+\underbrace{\sum_{\E\in \CT_h}\left[\int_\E\Delta (\PiK  w_h)(e-e_I)+\sum_{\ell\in \CE_{\E}}\int_{\ell}J_{\ell}(e-e_I)\right]}_{T_{2}}\\
&-\underbrace{\sum_{\E\in \CT_h}a^{\E}( w_h-\PiK  w_h, e-e_I)}_{T_{3}}+\underbrace{a_h(w_h,e_I)-a( w_h,e_I)}_{T_{4}},
\end{split}
\end{equation}
the last equality thanks to Lemma~\ref{ext2}.
Next, we bound each term $T_{i}$ separately.

For $T_1$, we use the definition of $b(\cdot,\cdot)$,
the fact that  $\|w\|_{0,\G_0}=\|w_h\|_{0,\G_0}=1$, a trace theorem and \eqref{poincare1} to write
\begin{align}\label{cota3}
T_{1}&=\l+\l_h-(\l+\l_h)\int_{\Go} ww_h=\dfrac{\l+\l_h}{2}\|e\|_{0,\G_0}^2\leq C\dfrac{\l+\l_h}{2}\|e\|_{0,\G_0}|e|_{1,\O}.
\end{align}

For $T_2$, first, we use a local trace inequality
(see \cite[Lemma~14]{BMRR}) and  \eqref{cotainter} to write
\begin{align*}
\|e-e_I\|_{0,\ell}&\leq C\!\left( h_\E^{-1/2}\|e-e_I\|_{0,\E}+h_\E^{1/2}|e-e_I|_{1,\E}\right)\leq Ch_{\E}^{1/2}\|e\|_{1,\E}.
\end{align*}
Hence, using \eqref{cotainter} again, we have
\begin{align}\label{cot}
\nonumber
T_{2}&\leq C\sum_{\E\in \CT_h}\left[\|\Delta (\PiK  w_h)\|_{0,\E}\|e-e_I\|_{0,\E}+\sum_{\ell \in \CE_{\E}}\|J_{\ell}\|_{0,\ell}\|e-e_I\|_{0,\ell}\right]\\\nonumber
&\leq C\sum_{\E\in \CT_h}\left[h_\E\|\Delta (\PiK  w_h)\|_{0,\E}\|e\|_{1,\E}+\sum_{\ell \in \CE_{\E}}h_\E^{1/2}\|J_{\ell}\|_{0,\ell}\|e\|_{1,\E}\right]\\
&\leq C\left\{\sum_{\E\in \CT_h}\left[h_\E^2\|\Delta (\PiK  w_h)\|_{0,\E}^2+\sum_{\ell \in \CE_{\E}}h_\E\|J_{\ell}\|_{0,\ell}^2\right]\right\}^{1/2}|e|_{1,\O},
\end{align}
where for the last estimate we have used  \eqref{poincare1}.

To bound  $T_3$, we use the \textit{stability} property \eqref{24}
and \eqref{cotainter} to write
\begin{align}\label{cota5}
T_{3}& \leq C\sum_{\E\in \CT_h}a_h^\E(w_h-\PiK  w_h,w_h-\PiK  w_h)^{1/2}\|e\|_{1,\E} \leq C \left(\sum_{\E\in \CT_h}\theta_{\E}^{2}\right)^{1/2}|e|_{1,\O},
\end{align}
where for the last estimate we have used Remark \ref{RKK} and \eqref{poincare1} again.

Finally, to bound  $T_4$, we add and subtract $\PiK  w_h$ on each $K\in\CT_h$
and use the  \textit{$k$-consistency} property \eqref{consistencia}:
\begin{align}\label{cota4}
\nonumber
T_{4}&=\sum_{\E\in \CT_h}\big[ a_h^\E(w_h-\PiK  w_h,e_I)-a^\E(w_{h}-\PiK  w_h,e_I)\big]\\\nonumber
& \leq\sum_{\E\in \CT_h}a_h^\E(w_h-\PiK  w_h,w_h-\PiK  w_h)^{1/2}a_{h}^\E(e_{I},e_I)^{1/2}\\\nonumber
&\quad+\sum_{\E\in \CT_h}a^\E(w_h-\PiK  w_h,w_h-\PiK  w_h)^{1/2}a^\E(e_{I},e_I)^{1/2}\\\nonumber
& \leq C\sum_{\E\in \CT_h}a_h^\E(w_h-\PiK  w_h,w_h-\PiK  w_h)^{1/2}|e_{I}|_{1,\E}\\
& \leq C \left(\sum_{\E\in \CT_h}\theta_{\E}^{2}\right)^{1/2}|e|_{1,\O},
\end{align}
where we have used the \textit{stability} property  \eqref{24}, \eqref{cotainter}
and \eqref{poincare1} for the last two inequalities. 
 
Thus, the result follows from \eqref{cotasalto}--\eqref{cota4}.
\end{proof}

Although the virtual approximate eigenfunction is $w_{h}$, this function is not known in practice. Instead of $w_{h}$, what can be used as an approximation of the eigenfunction is $\Pi_h w_{h}$, where $\Pi_h$ is  defined for $v_{h}\in V_{h}$ by
\begin{equation*}
(\Pi_h v_{h})|_{\E}:=\PiK v_{h}\quad\forall\E\in\CT_h.
\end{equation*}
Notice that $\Pi_h w_{h}$ is actually computable.
The following result shows that an estimate similar to that of Theorem \ref{erroipo} holds true for $\Pi_h w_{h}$.
\begin{corollary}
\label{corolario2}
There exists a constant $C>0$ independent of $h$ such that
\begin{equation*}
 |w-w_h|_{1,\O}+|w-\Pi_h  w_{h}|_{1,h}\leq C\left( \eta+ \dfrac{\l+\l_h}{2}\|w-w_{h}\|_{0,\G_0}\right).
\end{equation*}
 \end{corollary}
\begin{proof}
For each polygon $\E\in\CT_h$,  we have that
\begin{align*}
|w-\PiK  w_{h}|_{1,\E}\leq |w- w_h|_{1,\E}+|w_h-\PiK  w_h|_{1,\E}.
\end{align*}
Then, summing over all polygons we obtain 
\begin{align*}
|w-\Pi_h  w_{h}|_{1,h}
&\leq C\left( \sum_{\E\in \CT_h}|w- w_h|_{1,\E}^2+\sum_{\E\in \CT_h}|w_h-\PiK  w_h|_{1,\E}^2\right)^{1/2}.
\end{align*} 
Now, using \eqref{20} together with Remark~\ref{RKK},
we have that 
$$|w_h-\PiK  w_h|_{1,\E}^{2}\leq \dfrac{1}{c_{0}}S^\E(w_h-\PiK  w_h,w_h-\PiK  w_h)= \dfrac{1}{c_{0}}\theta_{\E}^{2}\leq \dfrac{1}{c_{0}}\eta_{\E}^{2}.$$
Thus, the result follows from Theorem~\ref{erroipo}.
\end{proof}

In what follows, we prove a convenient upper bound for the eigenvalue approximation. 
\begin{corollary}
\label{cotalambda}
There exists a constant $C>0$ independent of $h$ such that
 \begin{align*}
|\l-\l_h|\leq C\left( \eta+ \dfrac{\l+\l_h}{2}\|w-w_{h}\|_{0,\G_0}\right)^{2}.
  \end{align*} 
  \end{corollary}
  \begin{proof}
From the symmetry of the bilinear forms together with the facts
that $a(w,v)=\l b(w,v)$ for all $v\in \HuO$, $a_h(w_h,v_h)=\l_h b(w_h,v_h)$
for all $v_h\in \Vh$ and $b(w_h,w_h)=1$, we have
\begin{align}\label{lamdass}
\nonumber
|\l-\l_h|&=\dfrac{|a(w-w_h,w-w_h)-\l b(w-w_h,w-w_h)+a_h(w_h,w_h)-a(w_h,w_h)|}{b(w_h,w_h)}\\\nonumber
&\leq C\left[|w-w_h|_{1,\O}^2+\|w-w_h\|_{0,\G_0}^2+|a_h(w_h,w_h)-a(w_h,w_h)|\right]\\
&\leq C\left[|w-w_h|_{1,\O}^2+|a_h(w_h,w_h)-a(w_h,w_h)|\right],
\end{align}
where we have also used a trace theorem and \eqref{poincare1}. We now bound the last
term on the right-hand side above using the definition of
$a_{h}(\cdot,\cdot)$ and  \eqref{20}:
\begin{align*}
& \left|a_h(w_h,w_h)-a(w_h,w_h)\right|
\\
& \,\,
=\left|\sum_{\E\in\CT_h}\left[a^\E(\PiK w_h,\PiK w_h)
+S^\E\big(w_h-\PiK w_h,w_h-\PiK w_h\big)\right]
-\sum_{\E\in\CT_h}a^\E(w_h,w_h)\right|
\\
& \,\,
\leq\left|\sum_{\E\in\CT_h}
\left[a^\E\big(\PiK w_h,\PiK w_h\big)
-a^\E(w_h,w_h)\right]\right|+\sum_{\E\in\CT_h}
c_1\,a^\E\big(w_h-\PiK w_h,w_h-\PiK w_h\big)
\\
& \,\,
=\sum_{\E\in\CT_h}\left(1+c_1\right)
a^\E\big(w_h-\PiK w_h,w_h-\PiK w_h\big)\\
& \leq \left(1+c_1\right)\sum_{\E\in\CT_h}\left(\left|w_h-w\right|_{1,\E}^{2}
+\left|w-\PiK w_{h}\right|_{1,\E}^{2}\right).
\end{align*}
Finally, from the above estimate and \eqref{lamdass}  we obtain
\begin{align}
\label{cfutu}
|\l-\l_h|&\leq C\left(|w-w_h|_{1,\O}^{2}+|w-\Pi_h  w_{h}|_{1,h}^{2}\right).
\end{align} 
Hence, we conclude the proof thanks to Corollary~\ref{corolario2}.
\end{proof}
According to \eqref{eq29} and \eqref{eq293}, it seems reasonable to expect the term $\|w-w_{h}\|_{0,\G_0}$
in the estimate of Theorem~\ref{erroipo} to be of higher order than $|w-w_{h}|_{1,\O}$ and hence asymptotically negligible. However this cannot be rigorously  derived from \eqref{eq29} and \eqref{eq293}, which are only upper error bounds. In fact, the actual error $|w-w_{h}|_{1,\O}$ could be in principle of higher order than the estimate \eqref{eq29}.  
  
Our next goal is to prove that the term $\|w-w_{h}\|_{0,\G_0}$ is actually asymptotically negligible in the estimates of Corollaries \ref{corolario2} and \ref{cotalambda}. With this aim, we will modify the estimate \eqref{eq293} and prove that 
\begin{equation}
\label{ro1}
\|w-w_{h}\|_{0,\G_{0}}\leq Ch^{\min\{r,1\}/2}\left(|w-w_{h}|_{1,\O}+|w-\Pi_hw_{h}|_{1,h}\right).
\end{equation}
This proof is based on the arguments used in Section 4 from \cite{MRR2015}. To avoid repeating them step by step, in what follows we will only report the changes that have to be made in order to prove \eqref{ro1}.

We define in $\HuO$ the bilinear form $\widehat{a}(\cdot,\cdot):=a(\cdot,\cdot)+b(\cdot,\cdot)$, which is elliptic \cite[Lemma 2.1]{MRR2015}.
Let $u\in \HuO$ be the solution of 
$$\widehat{a}(u,v)=b(w,v)\qquad \forall v\in \HuO.$$ 
Since $a(w,v)=\l b(w,v)$ we have that $u=w/(\l+1).$
We also define in $V_{h}$ the bilinear form  $\widehat{a}_{h}(\cdot,\cdot):=a_{h}(\cdot,\cdot)+b(\cdot,\cdot)$, which is elliptic uniformly in $h$ \cite[Lemma 3.1]{MRR2015}. 
Let $u_{h}\in V_{h}$ be the solution of 
\begin{equation}
\label{ro2}
\widehat{a}_{h}(u_{h},v_{h})=b(w,v_{h})\qquad \forall v_{h}\in V_{h}.
\end{equation}
The arguments in the proof of Lemma 4.3 from \cite{MRR2015} can be easily modified to prove that 
$$\|u-u_{h}\|_{0,\G_{0}}\leq Ch^{\min\{r,1\}/2}\left(|u-u_{h}|_{1,\O}+|u-\Pi_{h} u_{h}|_{1,h}\right).$$
Then, using this estimate in the proof of  Theorem 4.4 from \cite{MRR2015} yields 
\begin{equation}
\label{ro3}
\|w-w_{h}\|_{0,\G_{0}}\leq Ch^{\min\{r,1\}/2}\left(|u-u_{h}|_{1,\O}+|u-\Pi_{h} u_{h}|_{1,h}\right).
\end{equation}
Now, since as stated above $u=w/(\l+1)$, we have that
\begin{equation}
\label{ro4}
|u-u_{h}|_{1,\O}\leq \dfrac{|w-w_{h}|_{1,\O}}{|\l+1|}+\left|\dfrac{1}{\l+1}-\dfrac{1}{\l_{h}+1}\right||w_{h}|_{1,\O}+\left|\dfrac{w_{h}}{\l_{h}+1}-u_{h}\right|_{1,\O}.
\end{equation}
For the second term on the right hand side above, we use \eqref{cfutu} to write 
\begin{align}
\label{ro5}
\left|\dfrac{1}{\l+1}-\dfrac{1}{\l_{h}+1}\right|=\dfrac{|\l-\l_{h}|}{|\l+1||\l_{h}+1|}&\leq C \left(|w-w_{h}|_{1,\O}^{2}+|w-\Pi_h w_{h}|_{1,h}^{2}\right).
\end{align}
To estimate the third term we recall first that
$$\widehat{a}_{h}(w_{h},v_{h})=(\l_{h}+1)b(w_{h},v_{h})\qquad \forall v_{h}\in V_{h}.$$
Then, subtracting this equation divided by $\l_{h}+1$ from \eqref{ro2} we have that 
$$\widehat{a}_{h}\!\left(u_{h}-\dfrac{w_{h}}{\l_{h}+1},v_{h}\right)=b(w-w_{h},v_{h})\qquad \forall v_{h}\in V_{h}.$$
Hence, from the uniform ellipticity of $\widehat{a}_{h}(\cdot,\cdot)$ in $V_{h}$, we obtain
\begin{align*}
\left\|u_{h}-\dfrac{w_{h}}{\l_{h}+1}\right\|_{1,\O}^{2}&\leq C \|w-w_{h}\|_{0,\G_{0}}\left\|u_{h}-\dfrac{w_{h}}{\l_{h}+1}\right\|_{0,\G_{0}}\leq C \|w-w_{h}\|_{0,\G_{0}}\left\|u_{h}-\dfrac{w_{h}}{\l_{h}+1}\right\|_{1,\O}.
\end{align*}
Therefore
\begin{align}
\label{ro6}
\left\|u_{h}-\dfrac{w_{h}}{\l_{h}+1}\right\|_{1,\O}&\leq C \|w-w_{h}\|_{0,\G_{0}}\leq C \|w-w_{h}\|_{1,\O}\leq C |w-w_{h}|_{1,\O},
\end{align}
the last inequality because of Poincar\'e inequality \eqref{poincare1}.
Then, substituting \eqref{ro5} and  \eqref{ro6} into \eqref{ro4} we obtain 
\begin{align}
\label{ro7}
|u-u_{h}|_{1,\O}\leq C\left(|w-w_{h}|_{1,\O}+|w-\Pi_{h} w_{h}|_{1,h}\right).
\end{align}

For the other term on the right hand side of \eqref{ro3} we have
\begin{align}
\label{roex}
|u-\Pi_{h}u_{h}|_{1,h}\leq |u-u_{h}|_{1,\O}+|u_{h}-\Pi_{h}u_{h}|_{1,h},
\end{align}
whereas  
\begin{align*}
|u_{h}-\Pi_{h}u_{h}|_{1,h}&\leq \left|u_{h}-\dfrac{w_{h}}{\l_{h}+1}\right|_{1,\O}+\dfrac{|w_{h}-\Pi_{h}w_{h}|_{1,h}}{\l_{h}+1}+\left|\Pi_{h}\left(\dfrac{w_{h}}{\l_{h}+1}-u_{h}\right)\right|_{1,h}\\
&\leq  2\left|u_{h}-\dfrac{w_{h}}{\l_{h}+1}\right|_{1,\O}+\dfrac{|w-w_{h}|_{1,\O}}{\l_{h}+1}+\dfrac{|w-\Pi_{h}w_{h}|_{1,h}}{\l_{h}+1}\\
&\leq C \left(|w-w_{h}|_{1,\O}+|w-\Pi_{h}w_{h}|_{1,h}\right),
\end{align*}
where we have used \eqref{ro6} for the last inequality. Substituting this and estimate \eqref{ro7} into \eqref{roex}  we obtain
$$|u-\Pi_{h}u_{h}|_{1,h}\leq C \left(|w-w_{h}|_{1,\O}+|w-\Pi_{h}w_{h}|_{1,h}\right).$$
Finally, substituting the above estimate and \eqref{ro7} into \eqref{ro3}, we conclude the proof of the following result.
\begin{lemma}
\label{asintotico}
There exists $C>0$ independent of $h$ such that
\begin{equation*}
\|w-w_{h}\|_{0,\G_{0}}\leq Ch^{\min\{r,1\}/2}\left(|w-w_{h}|_{1,\O}+|w-\Pi_hw_{h}|_{1,h}\right).
\end{equation*}
\end{lemma}
Using this result, now it easy to prove that the term $\|w-w_{h}\|_{0,\G_0}$ in Corollaries \ref{corolario2} and \ref{cotalambda} is asymptotically negligible. In fact, we have the following result.
\begin{theorem}
There exist positive constants $C$ and $h_{0}$ such that, for all $h<h_{0}$, there holds
\begin{align}
\label{ro9}
&|w-w_{h}|_{1,\O}+|w-\Pi_{h}w_{h}|_{1,h}\leq C\eta;\\
&|\l-\l_{h}|\leq C \eta^{2}.\label{ro10}
\end{align}
\end{theorem} 
\begin{proof}
From Lemma \ref{asintotico} and Corollary \ref{corolario2} we have
$$|w-w_{h}|_{1,\O}+|w-\Pi_{h}w_{h}|_{1,h}\leq C\left(\eta+h^{\min\{r,1\}/2}\left(|w-w_{h}|_{1,\O}+|w-\Pi_hw_{h}|_{1,h}\right)\right).$$
Hence, it is straightforward to check that there exists $h_{0}>0$ such that for all $h<h_{0}$ \eqref{ro9} holds true.

On the other hand, from Lemma \ref{asintotico} and \eqref{ro9} we have that for all $h<h_{0}$
$$\|w-w_{h}\|_{0,\G_{0}}\leq Ch^{\min\{r,1\}/2}\eta.$$
Then, for $h$ small enough, \eqref{ro10} follows from Corollary \ref{cotalambda} and the above estimate. 
\end{proof}
\subsection{Efficiency of the a posteriori error estimator}

We will show in this section  that the local error indicators
$\eta_{\E}$  are efficient in the sense of pointing
out which polygons should be effectively refined.

First, we prove an upper estimate of the volumetric residual term $R_{\E}$.
\begin{lemma}
\label{eficiencia1}
There exists a constant $C>0$ independent of $h_\E$ such that
\begin{equation*}
R_{\E}\leq C \left(|w-w_{h}|_{1,\E}+\theta_{\E}\right).
\end{equation*}
\end{lemma}
\begin{proof}
For any $\E\in\CT_h$, let $\psi_{\E}$ be the corresponding
interior bubble function. We define  $v:=\psi_{\E}\Delta (\PiK  w_h)$. 
Since  $v$ vanishes on the boundary of $\E$, it may be extended by zero to the whole domain $\O$.
This extension, again denoted by $v$, belongs to $\HuO$ and from Lemma~\ref{ext2} we have 
$$a^\E(e,v)=-a^{\E}\!\left(w_{h}-\PiK w_h,\psi_{\E}\Delta (\PiK  w_h)\right)
+\int_{\E}\Delta (\PiK  w_h) \psi_{\E}\Delta (\PiK  w_h).$$
Since  $\Delta (\PiK  w_h)\in \P_{k-2}(\E)$,
using Lemma~\ref{burbujainterior} and the above equality we obtain
\begin{align}\nonumber\label{ghtjy}
C^{-1}\|\Delta (\PiK  w_h)\|_{0,\E}^2&\leq \int_{\E}\psi_{\E}\Delta (\PiK  w_h)^2\\\nonumber
&=a^{\E}\!\left(e,\psi_{\E}\Delta (\PiK  w_h)\right)+a^{\E}\!\left(w_{h}-\PiK w_h,\psi_{\E}\Delta (\PiK  w_h)\right)\\\nonumber
&\leq C\left(\left|e\right|_{1,\E}+\left|w_{h}-\PiK w_h\right|_{1,\E}\right)\left|\psi_{\E}\Delta (\PiK  w_h)\right|_{1,\E}\\
&\leq Ch_\E^{-1}\left(\left|e\right|_{1,\E}+\theta_{\E}\right)\left\|\Delta (\PiK  w_h)\right\|_{0,\E},
\end{align}
where, for the  last inequality,  we have used again
Lemma~\ref{burbujainterior} and \eqref{20} together with Remark~\ref{RKK}.
Multiplying the above inequality by $h_{\E}$ allows us to conclude the proof.
\end{proof}

Next goal is to obtain an upper estimate
for the local term $\theta_{\E}.$
\begin{lemma}
\label{cotaR}
There  exists $C>0$ independent of $h_\E$ such that
\begin{equation*}
\theta_{\E}\leq C\left(\vert w-w_h\vert_{1,\E}+\vert w-\PiK  w_{h}\vert_{1,\E}\right).
\end{equation*}
\end{lemma}
\begin{proof}
From the definition of $\theta_{\E}$
together with Remark~\ref{RKK} and estimate \eqref{20} we have 
\begin{align*}
\theta_{\E}&\leq C\vert w_h-\PiK  w_h\vert_{1,\E}\leq  C\left(\vert w_h-w\vert_{1,\E}+\vert w-\PiK  w_{h}\vert_{1,\E}\right).
\end{align*}
The proof is complete.
\end{proof}

The following lemma provides an upper estimate for
the jump terms of the local error indicator.

\begin{lemma}
There exists a constant $C>0$ independent of $h_\E$ such that
\label{lema4}
\begin{align}
\label{eqa}
h_\E^{1/2}\left\|J_{\ell}\right\|_{0,\ell}&\leq C\big(\vert w-w_{h}\vert_{1,\E}+\theta_{\E}\big)\hspace{3.4cm}\qquad\forall\ell\in\CE_{\E}\cap\CE_{\G_{1}}, \\\label{eqb}
h_\E^{1/2}\left\|J_{\ell}\right\|_{0,\ell}&\leq C\big(\vert w-w_{h}\vert_{1,\E}+\theta_{\E}+h_\E^{1/2}\left\|\l w -\l_h w_h\right\|_{0,\ell}\big)\hspace{0.5cm}\forall \ell\in\CE_{\E}\cap\CE_{\G_{0}},\\\label{eqc}
h_\E^{1/2}\left\|J_{\ell}\right\|_{0,\ell}&\leq C\sum_{\E'\in \omega_{\ell}}
\big(\vert w-w_{h}\vert_{1,\E'}+\theta_{\E'}\big)\hspace{2.3cm}\qquad \forall\ell\in\CE_{\E}\cap\CE_{\O},
\end{align}
where  $\omega_{\ell}:=\{\E'\in\CT_{h}: \ell\in \CE_{\E'}\}$.
\end{lemma}
\begin{proof}
First, for $\ell \in \CE_{\E}\cap\CE_{\G_{1}}$,
we extend $J_{\ell}\in\P_{k-1}(\ell)$ to the element $\E$ as in Remark~\ref{extencion}.
Let $\psi_{\ell}$ be the corresponding edge bubble function. We define $v:=J_{\ell}\psi_{\ell}$. Then, $v$ may be extended by zero to the whole domain $\O$.
This extension, again denoted by $v$, belongs to $ \HuO$ and from Lemma~\ref{ext2} we have  that
$$a^{\E}(e,v)=-a^{\E}(w_{h}-\PiK  w_{h},J_{\ell}\psi_{\ell})
+\int_{\E}\Delta\left(\PiK  w_h\right)J_{\ell}\psi_{\ell}+\int_{\ell}J_{\ell}^{2}\psi_{\ell}.$$
For $J_{\ell}\in\P_{k-1}(\ell)$,
from Lemma~\ref{burbuja} and the above equality we obtain 
\begin{align*}
C^{-1}\left\|J_{\ell}\right\|^2_{0,\ell}&\leq \int_{\ell} J_{\ell}^2\psi_{\ell}\leq
C\left[\left(\vert e\vert_{1,\E}+\vert w_{h}-\PiK  w_{h}\vert_{1,\E}\right)
\left|\psi_{\ell}J_{\ell}\right|_{1,\E}+\left\|\Delta (\PiK  w_h)\right\|_{0,\E}
\left\|J_{\ell}\psi_{\ell}\right\|_{0,\E}\right]\\
&\leq C\left[\left(\vert e\vert_{1,\E}+\vert w_{h}-\PiK w_{h}\vert_{1,\E}\right)
h_\E^{-1/2}\left\|J_{\ell}\right\|_{0,\ell}+h_\E^{-1}\left(\theta_{\E}
+\vert e\vert_{1,\E}\right)h_\E^{1/2}\left\|J_{\ell}\right\|_{0,\ell}\right]\\
&\leq Ch_\E^{-1/2}\left\|J_{\ell}\right\|_{0,\ell}\big(\vert e\vert_{1,\E}+\theta_{\E}\big),
\end{align*}
where  we have used again Lemma \ref{burbuja} together with estimate \eqref{ghtjy}.
Multiplying by $h_{\E}^{1/2}$ the above inequality allows us to conclude \eqref{eqa}.

Secondly, for $\ell\in\CE_{\E}\cap\CE_{\G_{0}}$, we extend $v:=J_{\ell}\psi_{\ell}$ to  $\HuO$ as in the previous case.
Taking into account that in this case $J_{\ell}\in\P_{k}(\ell)$
and $\psi_{\ell}$ is a quadratic bubble function in $\E$, from Lemma~\ref{ext2}
we obtain
$$a^{\E}(e,v)=\l\int_{\ell} wJ_{\ell}\psi_{\ell}-\l_{h}\int_{\ell} w_{h}J_{\ell}\psi_{\ell}
-a^{\E}\left(w_h-\PiK  w_h,J_{\ell}\psi_{\ell}\right)+\int_{\E}\Delta
\left(\PiK w_h\right)J_{\ell}\psi_{\ell}+\int_{\ell} J_{\ell}^{2}\psi_{\ell}.$$
Then, repeating the previous arguments  we obtain
\begin{equation*}
\left\vert\int_{\ell} J_{\ell}^{2}\psi_{\ell}\right\vert
\leq C\left[\left\vert\l_{h}\int_{\ell} w_{h}J_{\ell}\psi_{\ell}
-\l\int_{\ell} wJ_{\ell}\psi_{\ell}\right\vert+h_\E^{-1/2}\left\|J_{\ell}\right\|_{0,\ell}
\left(\theta_{\E}+|e|_{1,\E}\right)\right].
\end{equation*}
Hence, using Lemma~\ref{burbuja} and a local trace inequality we arrive at
\begin{align*}
\|J_{\ell}\|^2_{0,\ell}&\leq C\left[\left\|\l w -\l_h w_h\right\|_{0,\ell}\left\|\psi_{\ell}J_{\ell}\right\|_{0,\ell}+h_\E^{-1/2}\left(\theta_{\E}+|e|_{1,\E}\right)\|J_{\ell}\|_{0,\ell}\right]\\
&\leq Ch_\E^{-1/2}\|J_{\ell}\|_{0,\ell}\left(\theta_{\E}+|e|_{1,\E}+h_\E^{1/2}\|\l w -\l_h w_h\|_{0,\ell}\right),
\end{align*}
where we have used Lemma \ref{burbuja} again. Multiplying by $h_{\E}^{1/2}$ the above inequality yields $\eqref{eqb}$.

Finally, for $\ell \in \CE_{\E}\cap\CE_\O$, we extend
$v:=J_{\ell}\psi_{\ell}$ to $\HuO$ as above again. Taking into account that
$J_{\ell}\in\P_{k-1}(\ell)$ and $\psi_{\ell}$ is a quadratic bubble function in $\E$,
from Lemma~\ref{ext2} we obtain
$$a(e,v)=-\sum_{\E'\in \omega_{\ell}}a^{\E'}(w_{h}-\Pi_{k}^{\E'} w_h,J_{\ell}\psi_{\ell})+\sum_{\E'\in \omega_{\ell}}\int_{\E'}\Delta \left(\Pi_{k}^{\E'} w_h\right) J_{\ell}\psi_{\ell}+\sum_{\E'\in \omega_{\ell}}\int_{\ell} J_{\ell}^{2}\psi_{\ell}.$$
Then, proceeding analogously to the previous case we obtain
\begin{equation*}
\|J_{\ell}\|^2_{0,\ell}\leq Ch_{\E}^{-1/2}\|J_{\ell}\|_{0,\ell}
\left[\sum_{\E'\in \omega_{\ell}}(|e|_{1,\E'}+\theta_{\E'})\right].
\end{equation*}
Thus, the proof is complete.
\end{proof}

Now, we are in a position to prove an upper bound for the local error indicators $\eta_\E$. 
\begin{theorem}
\label{eficiencia}
There  exists $C>0$ such that
$$
\eta_{\E}^2\leq 
C\left[\displaystyle\sum_{\E'\in \omega_{\E}}\left(|w-\Pi_k^{\E'}w_{h}|_{1,\E'}^{2}+|w-w_{h}|_{1,\E'}^{2}+\displaystyle\sum_{\ell \in \CE_{\E}\cap\CE_{\Go}}h_\E\|\l w -\l_h w_h\|_{0,\ell}^2\right)\right],
$$
$\text{ where }\omega_{\E}:=\{\E'\in \CT_{h}: \E' \text{ and } \E\text{ share an edge}\}$.
\end{theorem}
\begin{proof}
It follows immediately from Lemmas~\ref{eficiencia1}--\ref{lema4}.
\end{proof}
According to the above theorem, the error indicators $\eta_{\E}^{2}$ provide lower bounds of the error terms $\sum_{\E'\in \omega_{\E}}
\left(|w-\Pi_k^{\E'}w_{h}|_{1,\E'}^{2}+|w-w_{h}|_{1,\E'}^{2}\right)$ in the neighborhood $\omega_{\E}$ of $\E$. For those elements $\E$ with an edge on $\G_{0}$, the term $h_\E\|\l w -\l_h w_h\|_{0,\ell}^2$ also appears in the estimate. Let us remark that it is reasonable to expect  this terms to be asymptotically negligible. In fact,  this is the case at least for the global estimator $\eta^{2}=\sum_{\E\in \CT_h}\eta_{\E}^2$ as is shown in the following result.
\begin{corollary}
There exists a constant $C>0$ such that
$$
\eta^2\leq C\left(|w-w_{h}|_{1,\O}^2+\vert w-\Pi_h w_{h}\vert_{1,h}^{2}\right). 
$$
\end{corollary}
\begin{proof}
From Theorem \ref{eficiencia} we have that 
$$
\eta^2\leq C\left(|w-w_{h}|_{1,\O}^2+\vert w-\Pi_h w_{h}\vert_{1,h}^{2}+h\|\l w -\l_h w_h\|_{0,\G_{0}}^2\right). 
$$
The last term on the right
hand side above is  bounded as follows: 
\begin{align*}
%\displaystyle\sum_{\E\in \CT_{h}}\sum_{\ell \in \CE_{\E}\cap\CE_{\Go}}h_\E\|\l w -\l_h w_h\|_{0,\ell}^2&\leq
\|\l w -\l_h w_h\|_{0,\G_{0}}^2
&\leq 2 \l^{2}\|w-w_h\|_{0,\G_0}^{2}+2|\l-\l_h|^{2},
\end{align*}
where we have used that $\|w_{h}\|_{0,\G_{0}}=1$. Now, by using a trace inequality and Poincar\'e inequality  \eqref{poincare1} we have
$$\|w-w_h\|_{0,\G_0}\leq C|w-w_h|_{1,\O}.$$
On the other hand, using the estimate \eqref{cfutu}, we have
\begin{align*}
|\l-\l_{h}|^{2}\leq (|\l|+|\l_{h}|)|\l-\l_{h}|\leq C\left(|w-w_h|_{1,\O}^{2}+|w-\Pi_h  w_{h}|_{1,h}^{2}\right).
\end{align*} 
Therefore,
$$
\eta^2\leq C\left(|w-w_{h}|_{1,\O}^2+\vert w-\Pi_h w_{h}\vert_{1,h}^{2}\right)
$$
and we conclude the proof.
\end{proof}

\setcounter{equation}{0}
\section{Numerical results}
\label{SEC:NUMERejemplo}
In this section, we will investigate the behavior of an adaptive scheme driven by the error indicator  in two numerical tests that
differ in the shape of the computational domain $\O$ and,
hence, in the regularity of the exact solution. With this aim, we have
implemented in a MATLAB code a lowest-order VEM ($k=1$) on arbitrary
polygonal meshes following the ideas proposed in \cite{BBMR2014}.

To complete the choice of the VEM, we had to choose  the bilinear forms
$S^{\E}(\cdot,\cdot)$ satisfying \eqref{20}. In this respect, we
proceeded as in \cite[Section 4.6]{BBCMMR2013}: for each polygon $\E$ with
vertices $P_1,\dots,P_{N_{\E}}$, we used
$$
S^{\E}(u,v):=\sum_{r=1}^{N_{\E}}u(P_r)v(P_r),
\qquad u,v\in V^{\E}_1.
$$

In all our tests we have initiated the adaptive process with a
coarse triangular mesh. In order to compare the performance of
VEM with that of a finite element method (FEM), we have used two different algorithms to refine the meshes.
The first one is based on  a classical FEM strategy for which all the subsequent  meshes consist of triangles. In such a case,
for $k=1$, VEM reduces to FEM. The other procedure to refine the
meshes is described in \cite{BMm2as}. It consists of splitting
each element into $n$ quadrilaterals ($n$ being the number of edges of the polygon) by connecting the
barycenter of the element with the midpoint  of each edge as shown
in Figure~\ref{FIG:cero} (see \cite{BMm2as} for more details).
Notice that although this process is initiated with a mesh of triangles,
the successively created meshes will contain other kind of convex  polygons, as can be seen in Figures~\ref{FIG:uno} and \ref{FIG:cinco}.

\begin{figure}[H]
\centering
\subfigure[Triangle $\E$ refined into 3 quadrilaterals.]{\includegraphics[height=3.6cm, width=3.6cm]{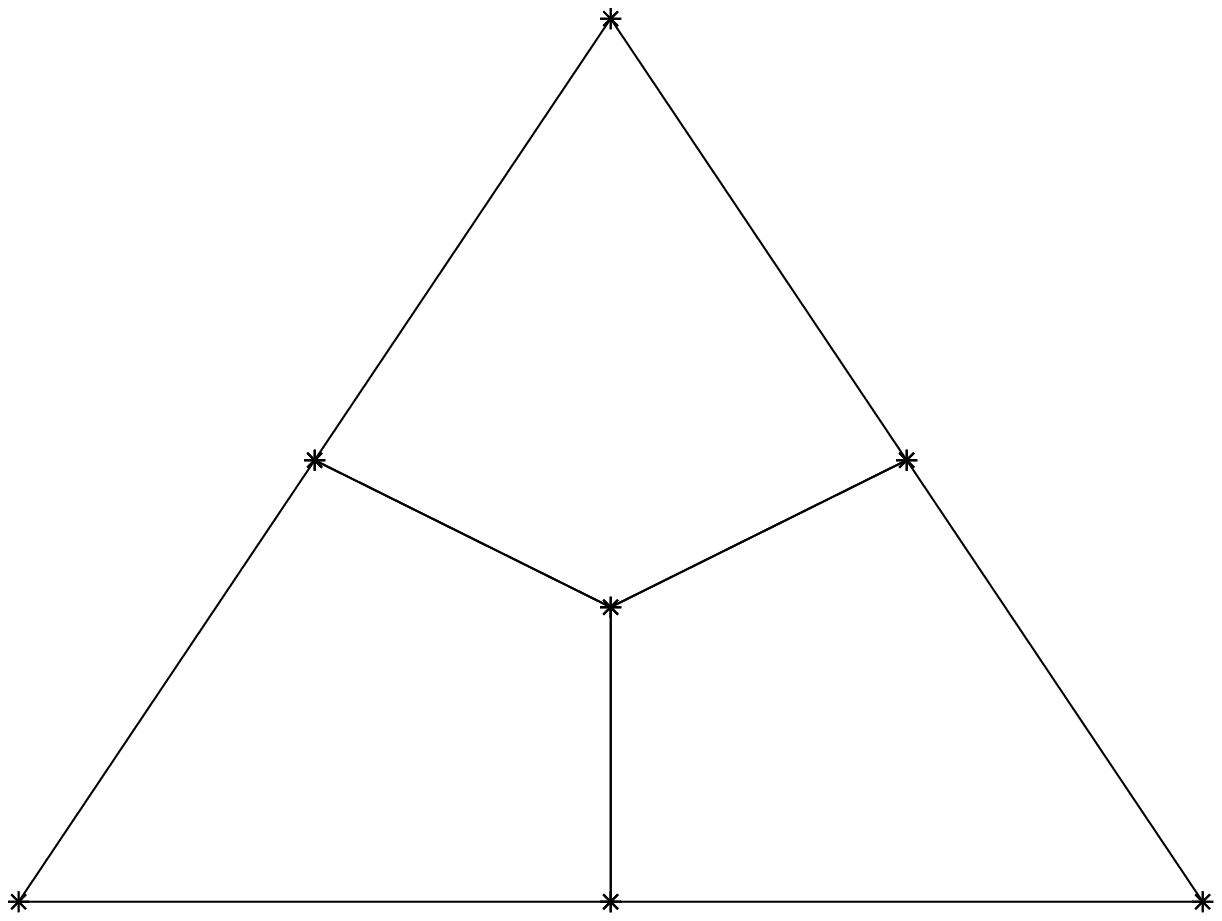}}
\hspace{0.5cm}\subfigure[Pentagon $\E$ refined into 5 quadrilaterals.]{\includegraphics[height=3.6cm, width=3.6cm]{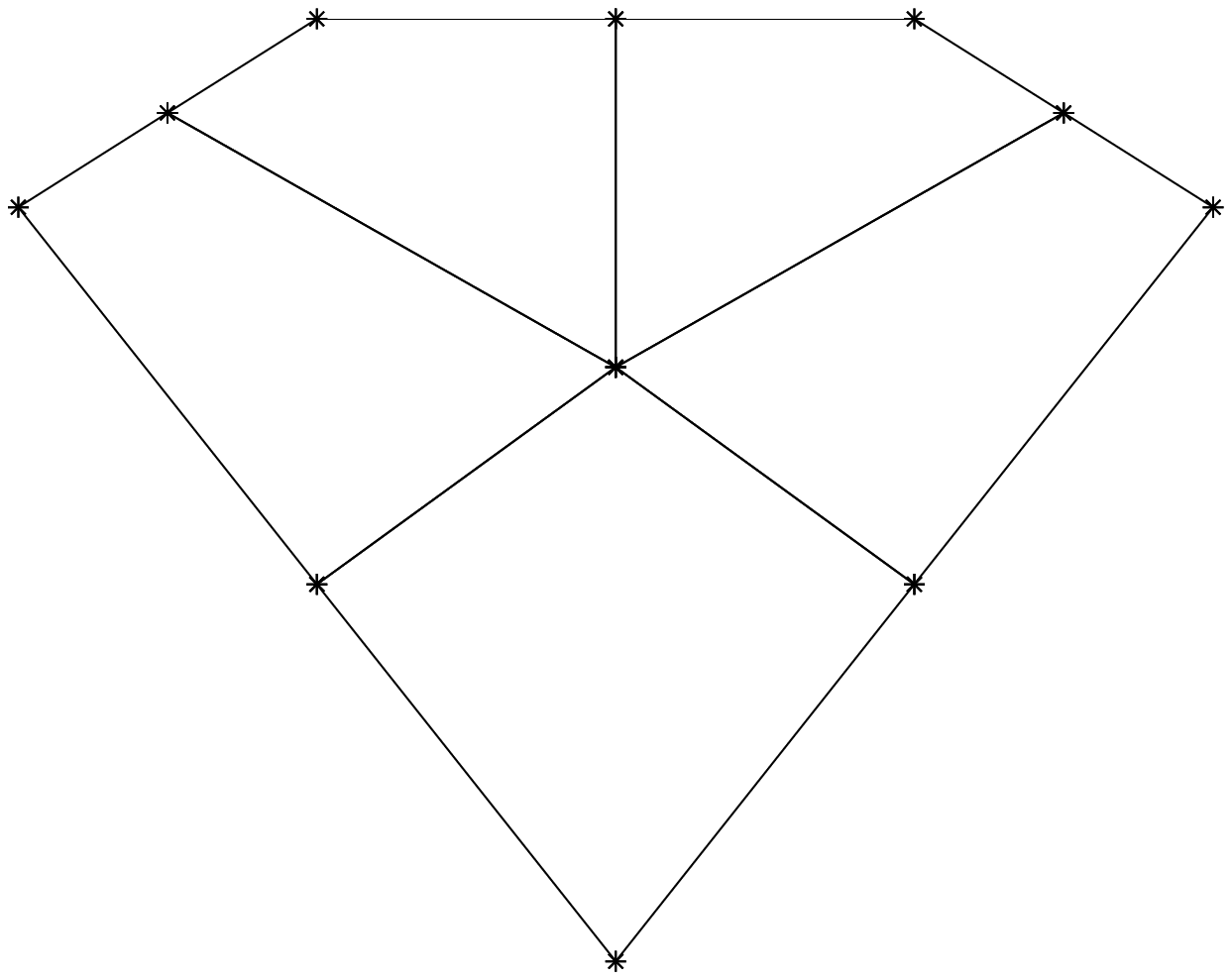}}
\caption{ Example of refined  elements for VEM strategy.}
 \label{FIG:cero}
\end{figure}

Since  we have chosen $k=1$, according to the definition
of the local virtual element space $V_{1}^{\E}$ (cf. \eqref{Vk}), the term
$R_{\E}^{2}:=h_{\E}^{2}\|\Delta w_{h}\|_{0,\E}^{2}$ vanishes.
Thus, the error indicators reduce in this case to 
$$\eta_{\E}^{2}= \theta_{\E}^2+\sum_{\ell \in \CE_{\E}}h_\E\|J_{\ell}\|_{0,\ell}^2
\qquad \forall\E\in \CT_{h}.$$
Let us remark that in the case of triangular meshes, the term
$\theta_{\E}^2:=a_h^\E(w_h-\PiK  w_h,w_h-\PiK  w_h)$ vanishes too,
since $V_{1}^{\E}=\P_{1}(\E)$ and hence $\PiK$ is the identity.
By  the same reason, the projection $\PiK$ also disappears
in the definition  \eqref{saltocal} of $J_{\ell}$.
Therefore, for triangular meshes, not only VEM reduces to FEM,
but also the error indicator becomes the classical well-known
edge-residual error estimator (see \cite{AP_APNUM2009}):
 \begin{equation*}
\eta_{\E}^2:=\disp\sum_{\ell \in \CE_{\E}}h_\E\|J_{\ell}\|_{0,\ell}^2 \quad\qquad  \text{with}\qquad\quad J_{\ell}:=\left\{\begin{array}{ll}
\dfrac{1}{2}\left[\!\!\left[ \dfrac{\partial w_{h}}{\partial{n}}\right]\!\!\right]_{\ell},\quad & \ell\in \CE_{\O},
\\[0.4cm]
\l_h w_h-\dfrac{\partial w_{h}}{\partial{n}}, \quad & \ell\in\CE_{\G_0},
\\[0.4cm]
-\dfrac{\partial w_{h}}{\partial{n}}, \quad & \ell\in\CE_{\G_1}.
\end{array}\right.
\end{equation*}
 In what follows, we report the results of a couple of tests.
 In both cases, we will restrict our attention to the approximation
 of the eigenvalues. Let us recall that according to  Corollary~\ref{cotalambda},
 the global error estimator $\eta^{2}$ provides an upper bound of
 the error of the computed eigenvalue.

\subsection{Test 1: Sloshing in a square domain.}
We have chosen for this test a problem with known analytical solution.
It corresponds to the computation of the sloshing modes of a
two-dimensional fluid contained in the domain
$\O:=(0,1)^2$ with a horizontal free surface
$\Go$ as shown in Figure~\ref{FIG:SLOSH}. The  solutions of this problem are
$$
\l_n=n\pi\tanh(n\pi),
\qquad w_n(x,y)=\cos(n\pi x)\sinh(n\pi y),
\qquad n\in\N.
$$

\vspace*{0.25cm}
\begin{figure}[H]
\begin{center}
\input{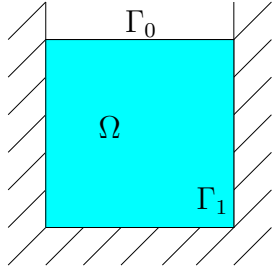}
\caption{Test 1. Sloshing in a square domain.}
\label{FIG:SLOSH}
\end{center}
\end{figure}
We have used the two refinement procedures  (VEM and FEM ) described above. Both schemes are based on the strategy of refining
those elements $\E$ which satisfy
$$\eta_{\E}\geq 0.5 \max_{\E'\in \CT_{h}}\{\eta_{\E'}\}.$$
 
Figures~\ref{FIG:uno} and \ref{FIG:dos}  show the adaptively refined meshes
obtained with VEM and FEM procedures, respectively. 
\begin{figure}[H]
\centering
\subfigure[Initial  mesh. ]{\includegraphics[height=5.0cm, width=5.0cm]{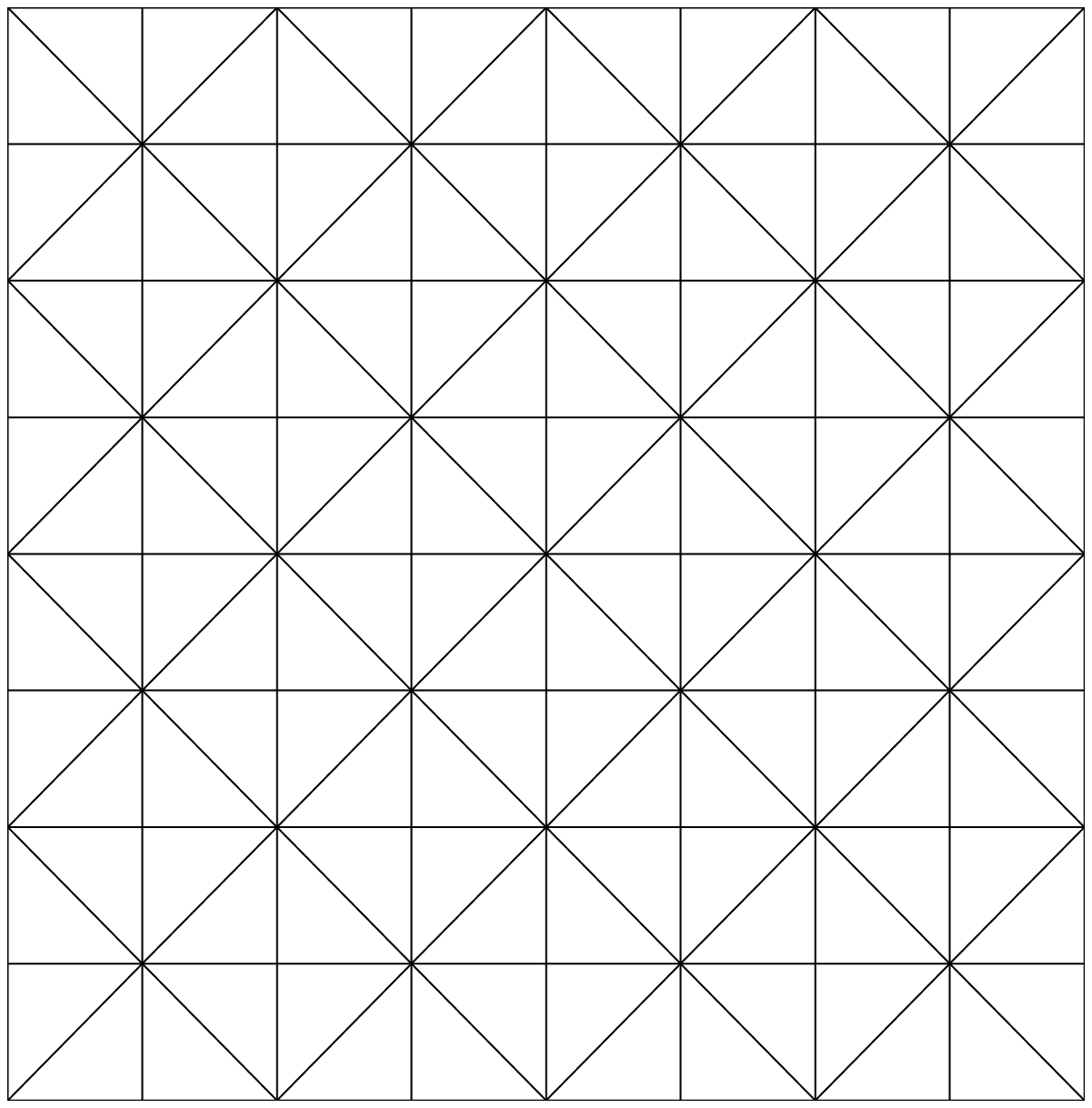}}
\subfigure[Step 1.]{\includegraphics[height=5.0cm, width=5.0cm]{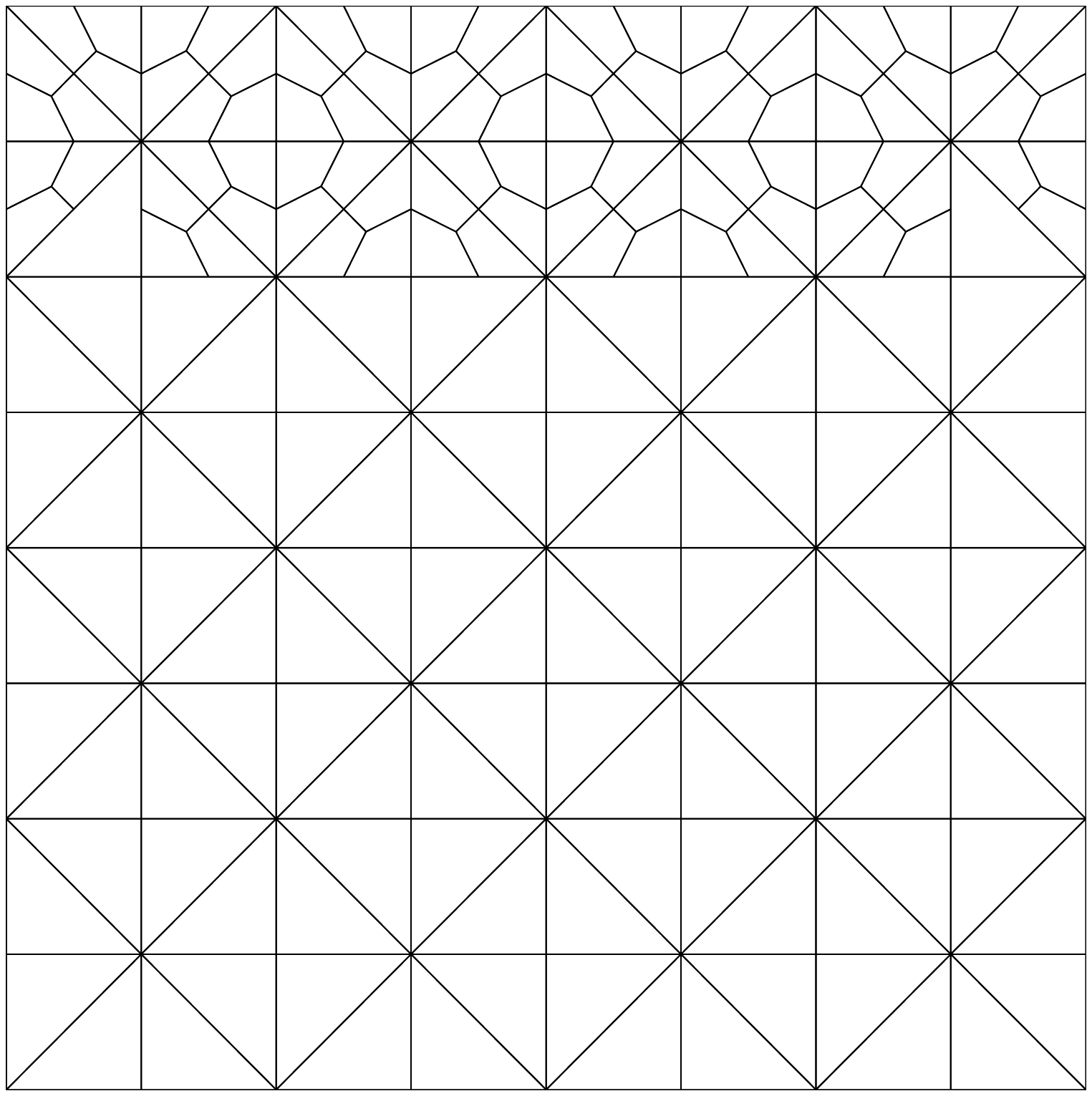}}\\
\subfigure[Step  3. ]{\includegraphics[height=5.0cm, width=5.0cm]{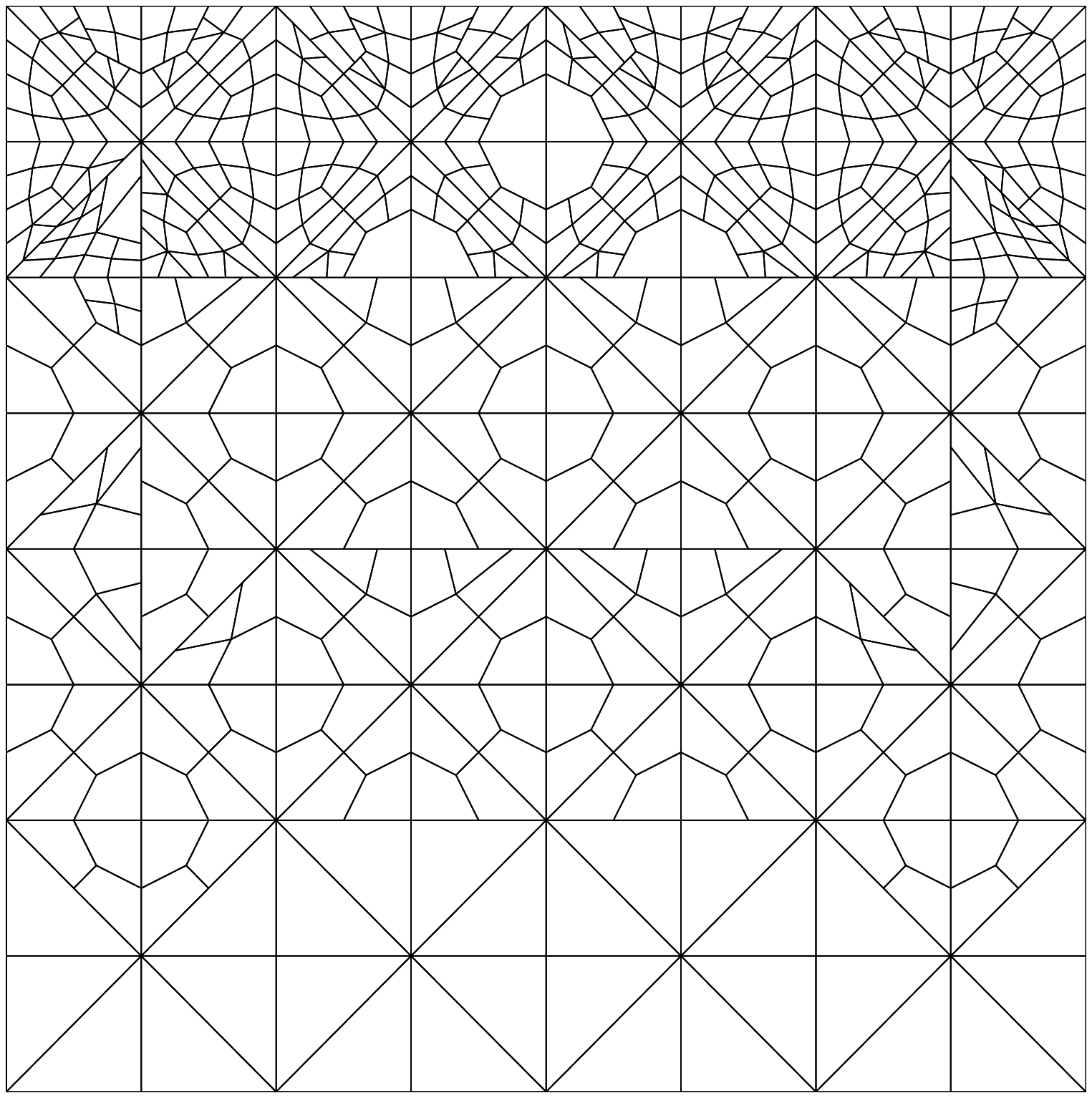}}
\subfigure[Step 6.]{\includegraphics[height=5.0cm, width=5.0cm]{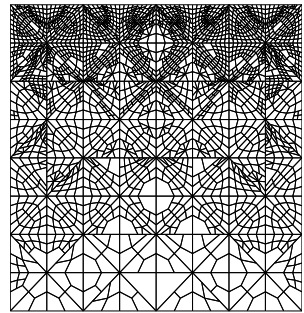}}
\caption{ Test 1. Adaptively refined meshes obtained with  VEM scheme at refinement steps 0, 1, 3 and 6.}
\label{FIG:uno}
\end{figure}
\begin{figure}[H]
\centering
\subfigure[Initial mesh.]{\includegraphics[height=5.0cm, width=5.0cm]{test1uno.eps}}
\subfigure[Step 1. ]{\includegraphics[height=5.0cm, width=5.0cm]{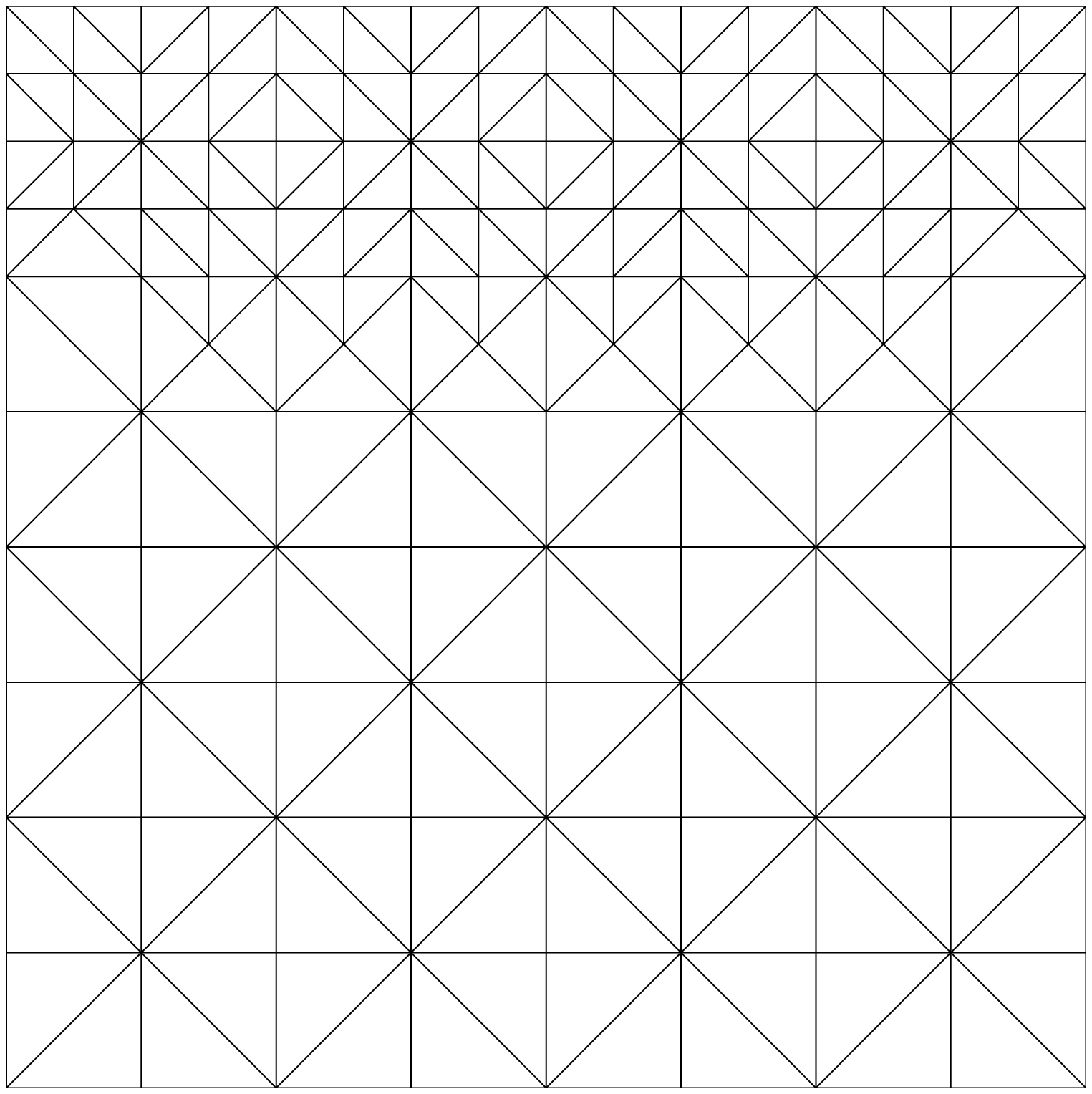}}\\
\subfigure[Step 3.]{\includegraphics[height=5.0cm, width=5.0cm]{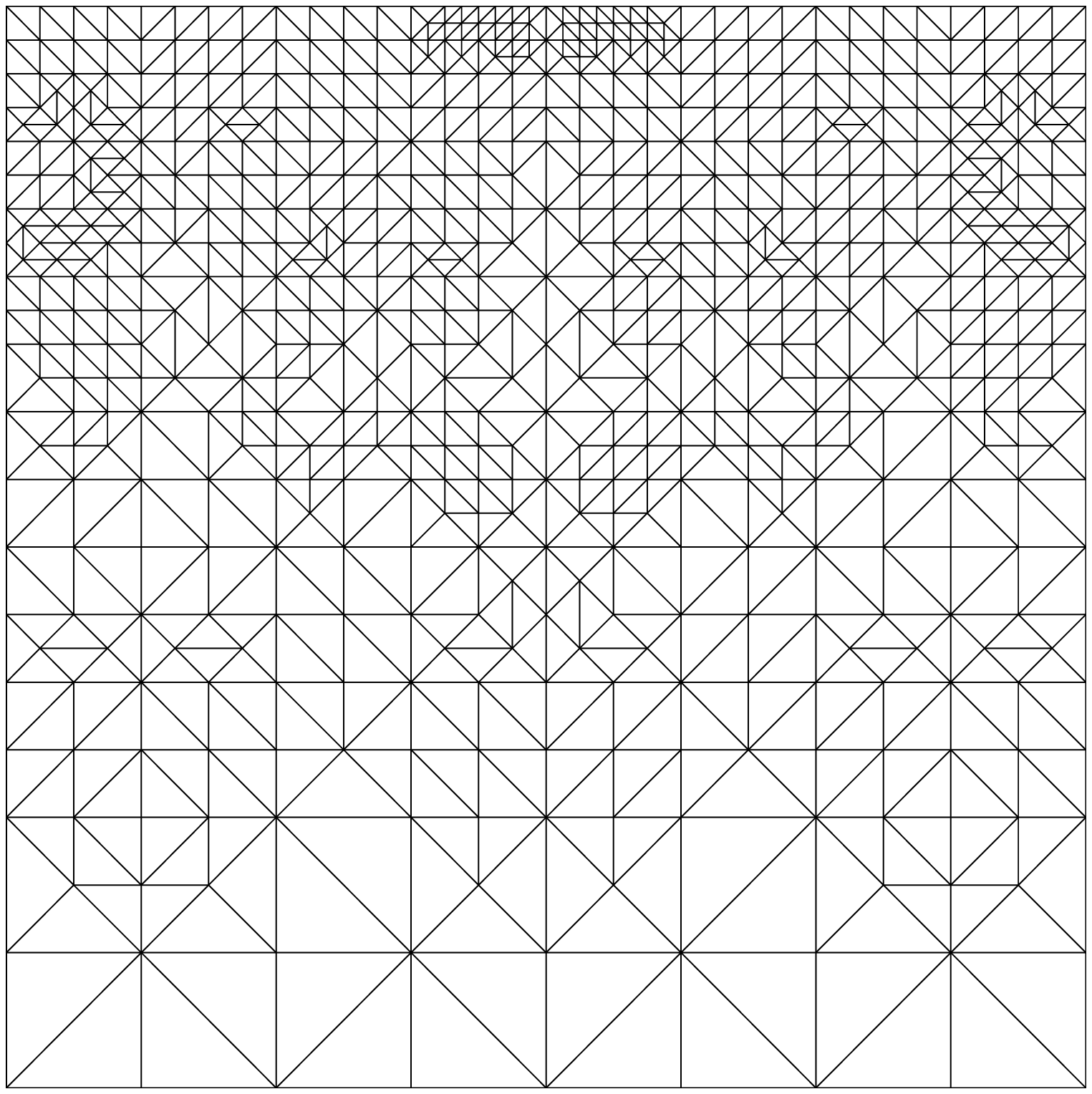}}
\subfigure[Step 6.]{\includegraphics[height=5.0cm, width=5.0cm]{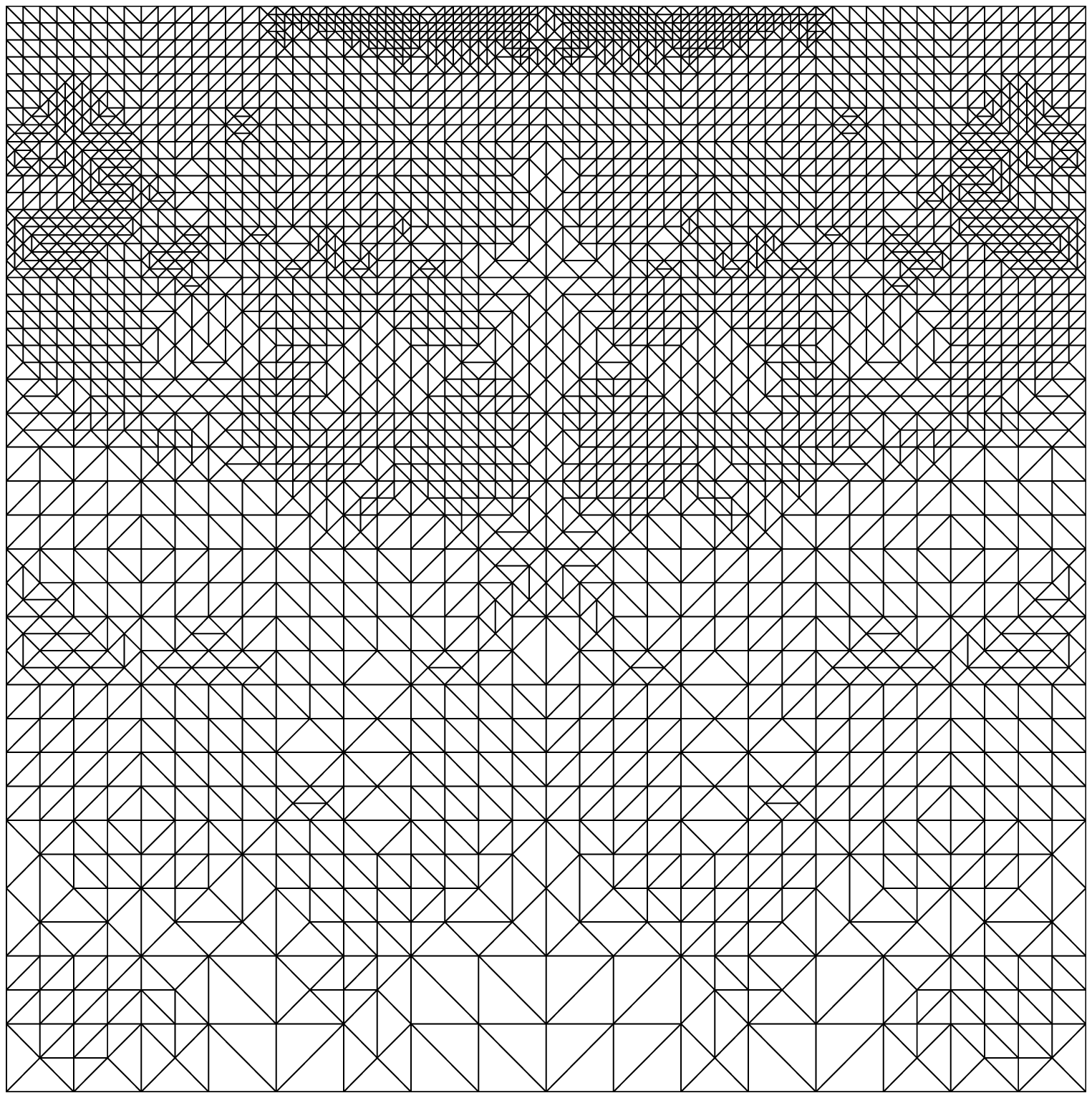}}
\caption{ Test 1. Adaptively refined meshes obtained with  FEM scheme at refinement steps 0, 1, 3 and 6.}
 \label{FIG:dos}
\end{figure}

Since the eigenfunctions of this problem are smooth, according to \eqref{eq29} we have that 
$|\l-\l_{h}|=\mathcal{O}(h^{2})$. Therefore, in case of uniformly refined meshes,
$|\l-\l_{h}|=\mathcal{O}\left(N^{-1}\right)$, where $N$ denotes the number
of degrees of freedom which is the optimal convergence rate that
can be attained.

Figure~\ref{FIG:tres} shows the error curves for the computed lowest eigenvalue on
uniformly refined meshes and adaptively refined meshes with FEM and VEM schemes.
The plot also includes a line of slope $-1$, which correspond to the optimal
convergence rate of the method $\mathcal{O}\left(N^{-1}\right)$. 

\begin{figure}[H]
\caption{Test 1. Error curves of  $|\l_1-\l_{h1}|$ for  uniformly refined meshes (``Uniform FEM''), adaptively refined meshes with  FEM (``Adaptive FEM'') and adaptively refined meshes with VEM (``Adaptive VEM'').}
\centering\includegraphics[height=10cm, width=10cm]{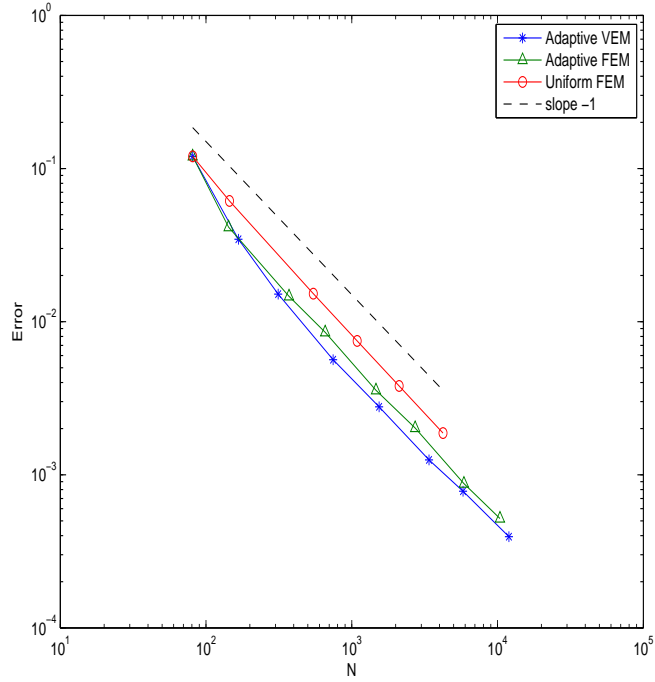}
\label{FIG:tres}
\end{figure}
It can be seen from Figure~\ref{FIG:tres} that the three refinement
schemes lead to the correct convergence rate. Moreover,
the performance of adaptive VEM is slightly better than that of 
adaptive FEM, while this is also better than uniform FEM.

We report in Table \ref{TABLA:N2}, the errors $|\l_{1}-\l_{h1}|$
and the estimators $\eta^{2}$ at each step of the adaptive VEM scheme.
We include in the table the terms $\theta^{2}:=\sum_{\E\in \CT_{h}}\theta_{\E}^{2}$
which arise from the inconsistency of VEM and
$J^{2}:=\sum_{\E\in \CT_{h}}\left(\sum_{\ell\in \CE_{\E}}h_{\E}\|J_{\ell}\|_{0,\ell}^{2}\right)$
which arise from the edge residuals. We also report in the table the effectivity indexes
$|\l_{1}-\l_{h1}|/\eta^{2}$.
\begin{table}[H]
\begin{center}
\caption{Test 1. Components of the error estimator and  effectivity indexes  on the adaptively refined meshes with VEM.}
\begin{tabular}{|c|c|c|c|c|c|c|}
    \hline
     $N$ & $\l_{h1}$  &$|\l_{1}-\l_{h1}|$ &$\theta^{2}$ &$J^{2}$ & $\eta^{2}$ & $\dfrac{|\l_{1}-\l_{h1}|}{\eta^{2}}$\\
\hline
38 &3.2499  &  0.1200  &       0   & 0.8245  &  0.8245  &  0.1456\\
167  &  3.1644 &   0.0345 &   0.0111&    0.2469  &  0.2580&    0.1339\\
313  &  3.1450  &  0.0151  &  0.0117  &  0.1108    &0.1225  &  0.1234\\
745  &  3.1355  &  0.0056   & 0.0054    &0.0427  &  0.0481   & 0.1171\\
1540  &  3.1327  &  0.0028  &  0.0033  &  0.0216  &  0.0249  &  0.1113\\
 3392  & 3.1311  &  0.0013   & 0.0015  &  0.0102  &  0.0117   & 0.1069\\
 5806  & 3.1307   & 0.0008  &  0.0009   & 0.0064    &0.0073    &0.1069\\
 11973 &  3.1303   & 0.0004  &  0.0005  &  0.0032  &  0.0037   & 0.1075\\
               \hline 
      \end{tabular}
\label{TABLA:N2}
\end{center}
\end{table}
It can be seen from Table \ref{TABLA:N2} that the effectivity
indexes are bounded above
and below far from zero and that the inconsistency and edge
residual terms are roughly speaking of the same order, none of
them being asymptotically negligible.

\subsection{Test 2:}

The aim of this test is to assess the performance of the adaptive
scheme when solving  a problem with a singular solution.
In this test $\O$ consists of a unit square from which
it is subtracted an equilateral triangle as shown in
Figure~\ref{FIG:SLOSH2}.  In this case $\O$ has a reentrant angle $\omega=\frac{5\pi}{3}$.
Therefore, the Sobolev exponent is $r_\O:=\frac{\pi}{\omega}=3/5$,
so that the eigenfunctions will belong to  $H^{1+r}(\O)$ for all $r< 3/5$,
but in general not to $H^{1+3/5}(\O)$.  Therefore, according
to \eqref{eq29}, using quasi-uniform  meshes,
the convergence rate for the eigenvalues should be
$|\l-\l_{h}|\approx \mathcal{O}\left(h^{6/5}\right)\approx \mathcal{O}\left(N^{-3/5}\right)$.
An efficient adaptive scheme should lead to refine the meshes
in such a way that the optimal order $|\l-\l_{h}|= \mathcal{O}\left(N^{-1}\right)$
could  be recovered.
\begin{figure}[H]
\begin{center}
\input{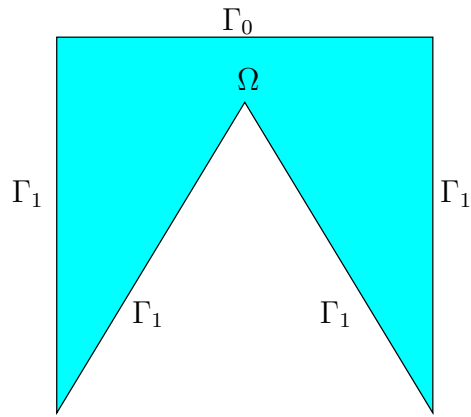}
\caption{Test 2. Domain $\Omega$.}
\label{FIG:SLOSH2}
\end{center}
\end{figure}

Figures~\ref{FIG:cinco} and \ref{FIG:seis} show the
adaptively refined meshes obtained with the VEM and FEM adaptive schemes, respectively.
 \begin{figure}[H]
 \vspace{-0.5cm}
\centering
\subfigure[Initial  mesh. ]{\includegraphics[height=4.6cm, width=4.6cm]{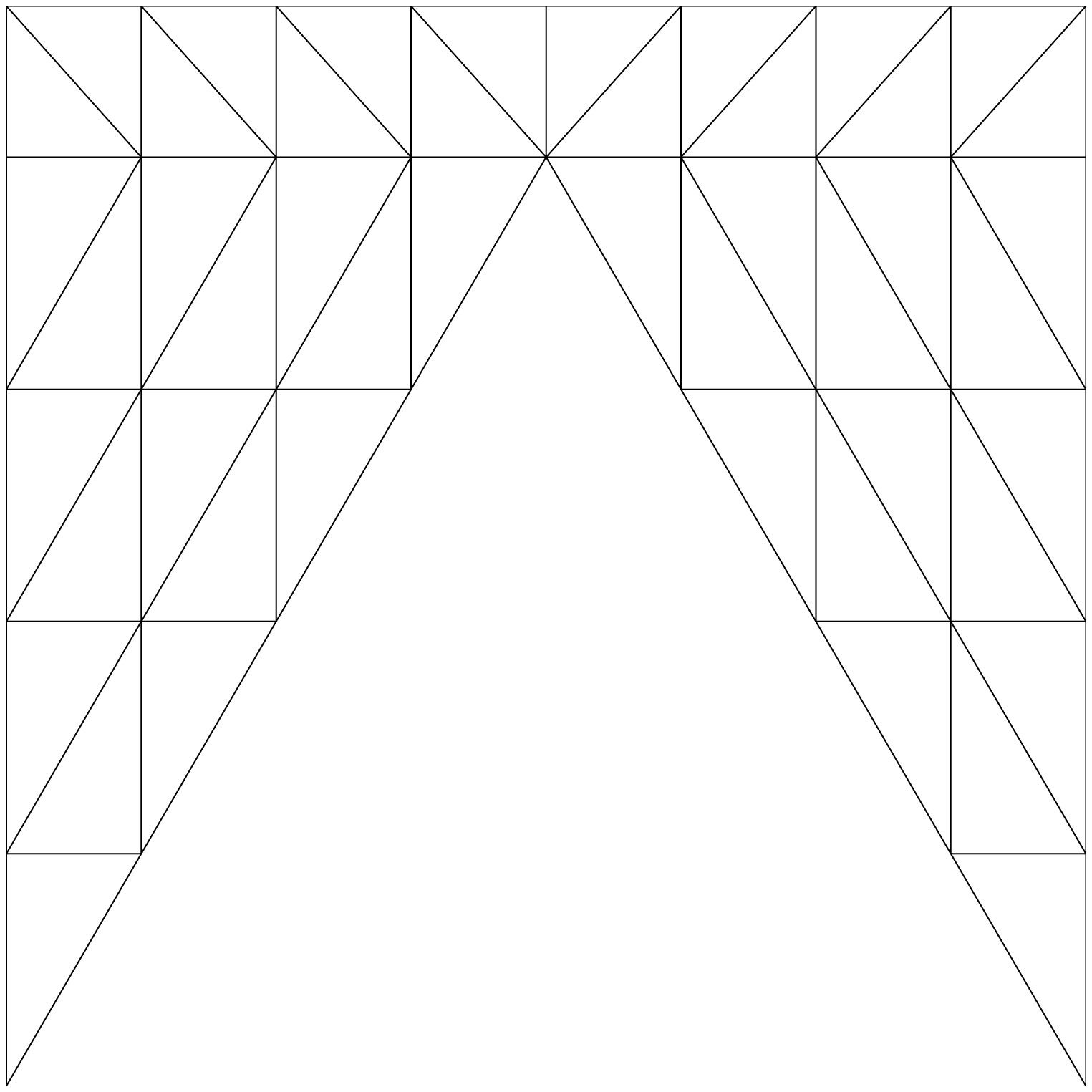}}
\subfigure[Step 1.]{\includegraphics[height=4.6cm, width=4.6cm]{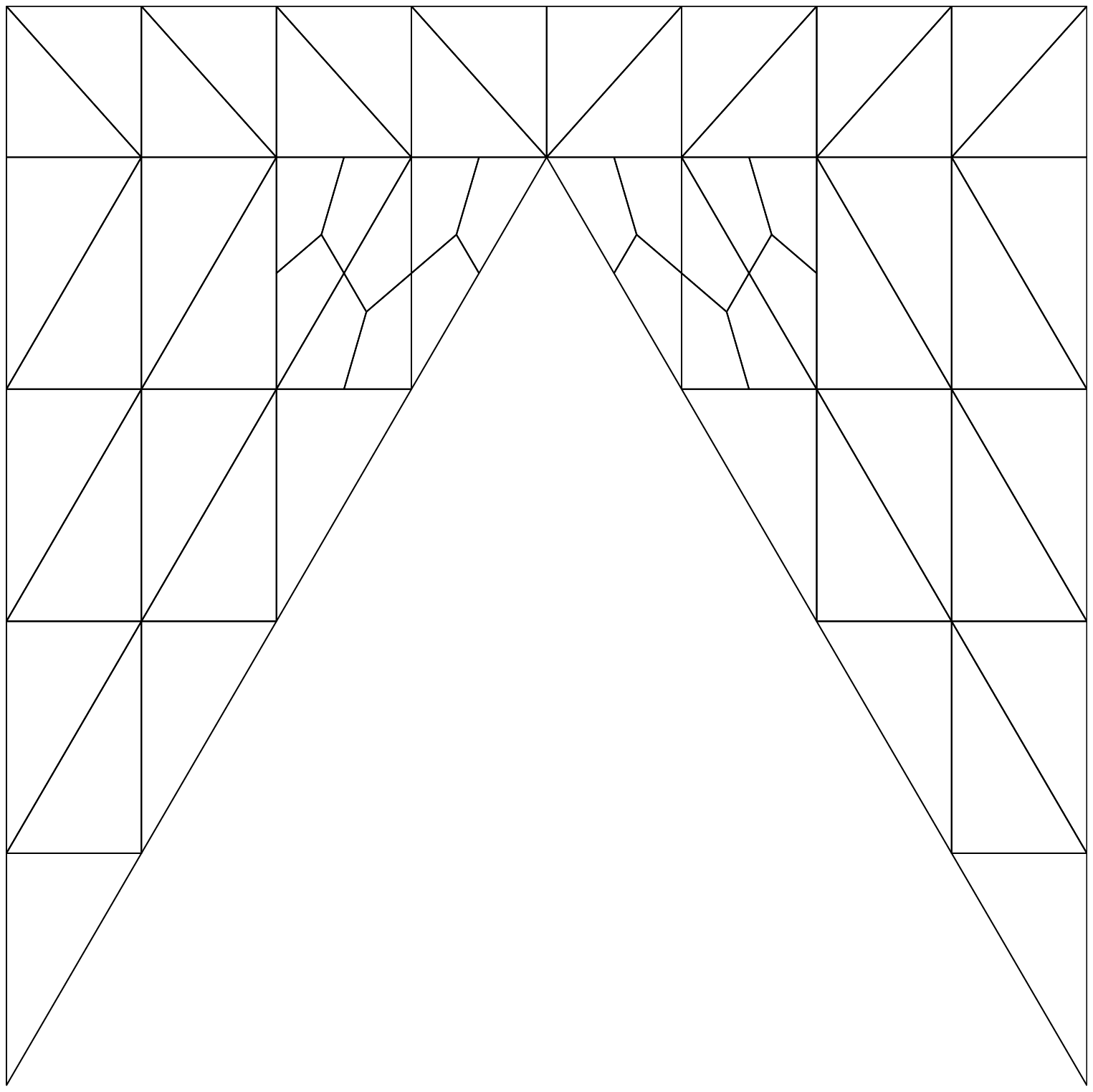}}\\\vspace{-0.3cm}
\subfigure[Step 4.]{\includegraphics[height=4.6cm, width=4.6cm]{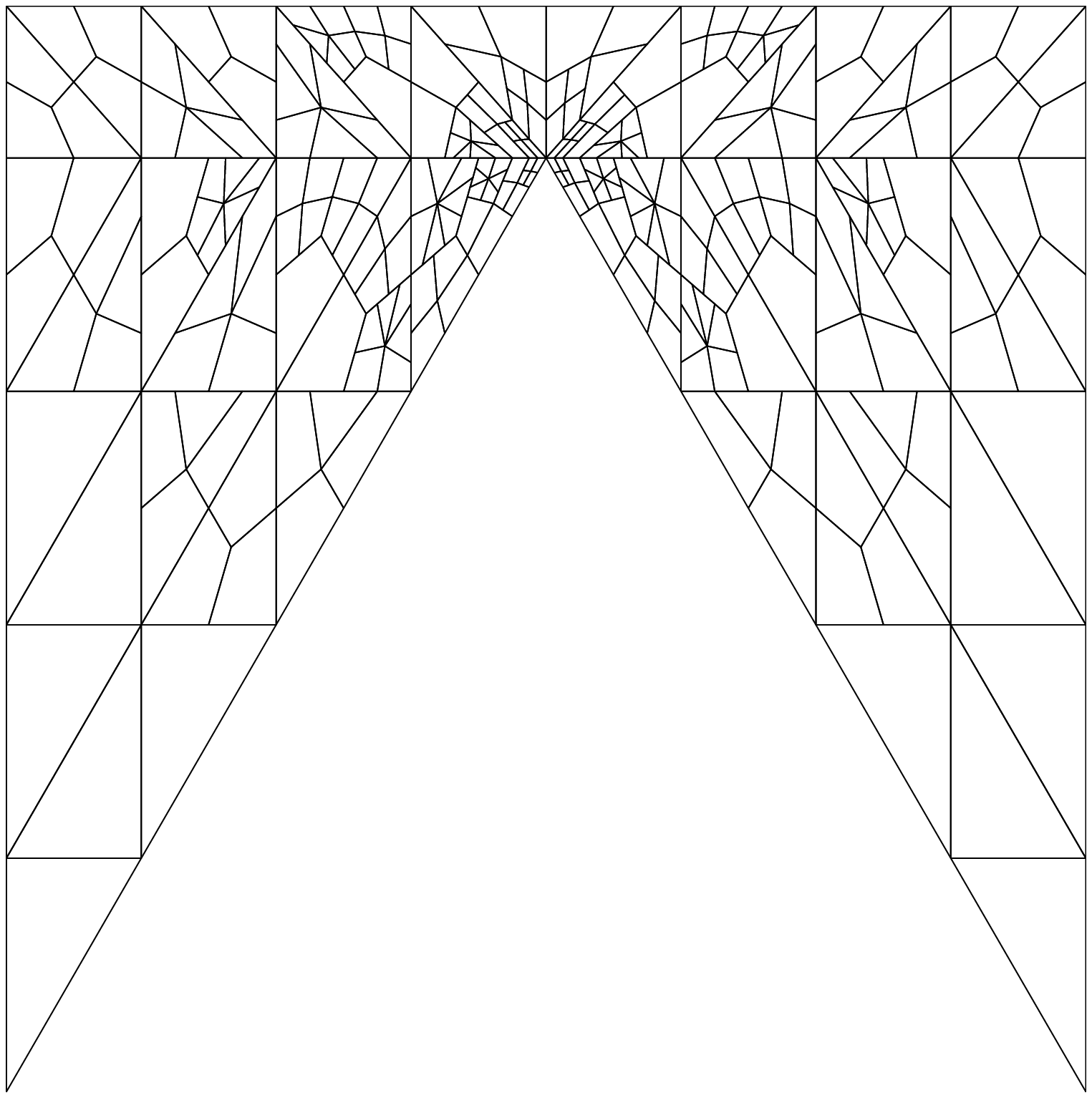}}
\subfigure[Step 6.]{\includegraphics[height=4.6cm, width=4.6cm]{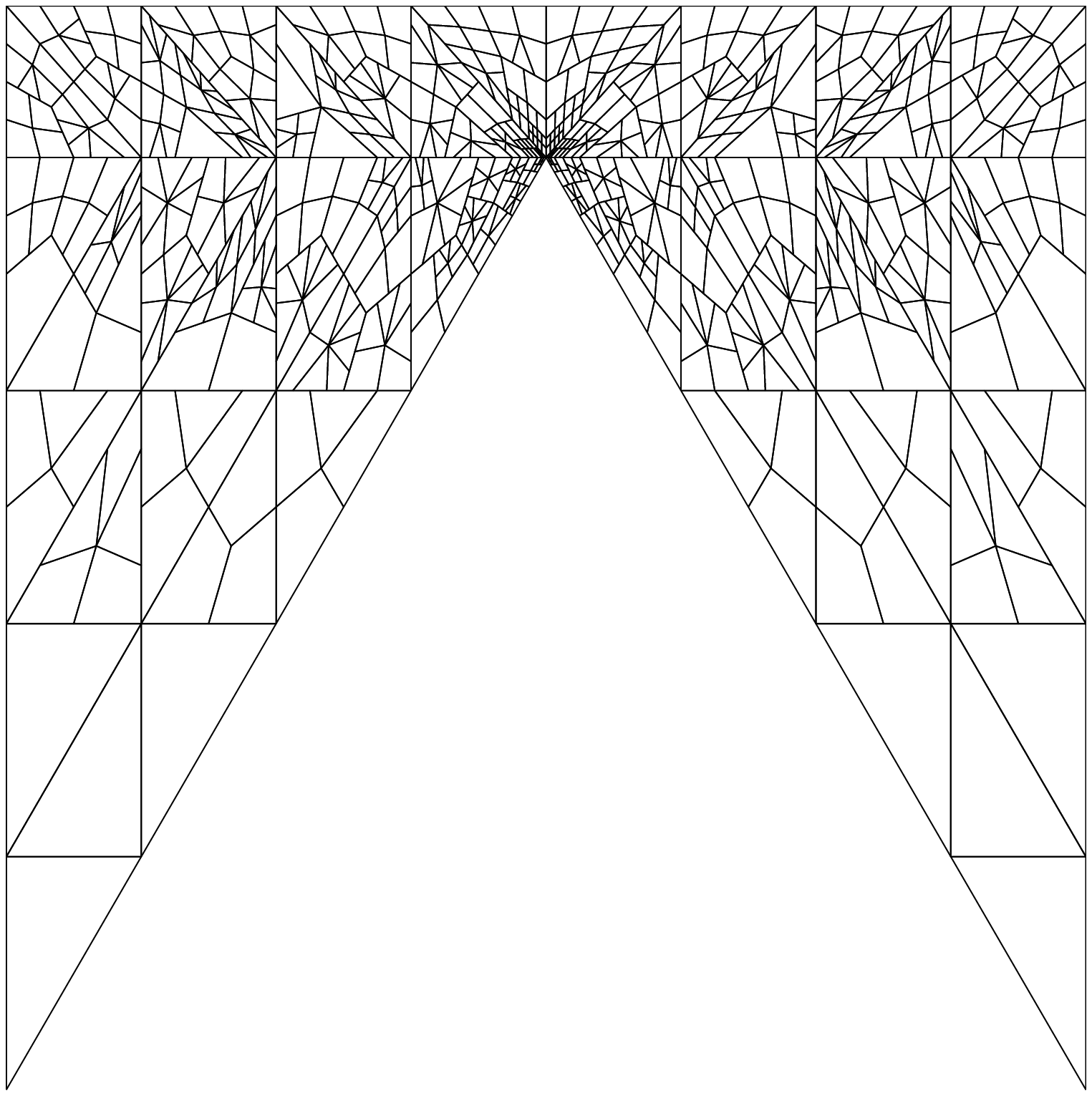}}
\vspace{-0.3cm}\caption{ Test 2. Adaptively refined meshes obtained with  VEM scheme at refinement steps 0, 1, 4 and 6.}
 \label{FIG:cinco}
\end{figure}
\begin{figure}[H]
\vspace{-0.3cm}
\centering
\subfigure[Initial mesh.]{\includegraphics[height=4.6cm, width=4.6cm]{test2uno.eps}}
\subfigure[Step 1. ]{\includegraphics[height=4.6cm, width=4.6cm]{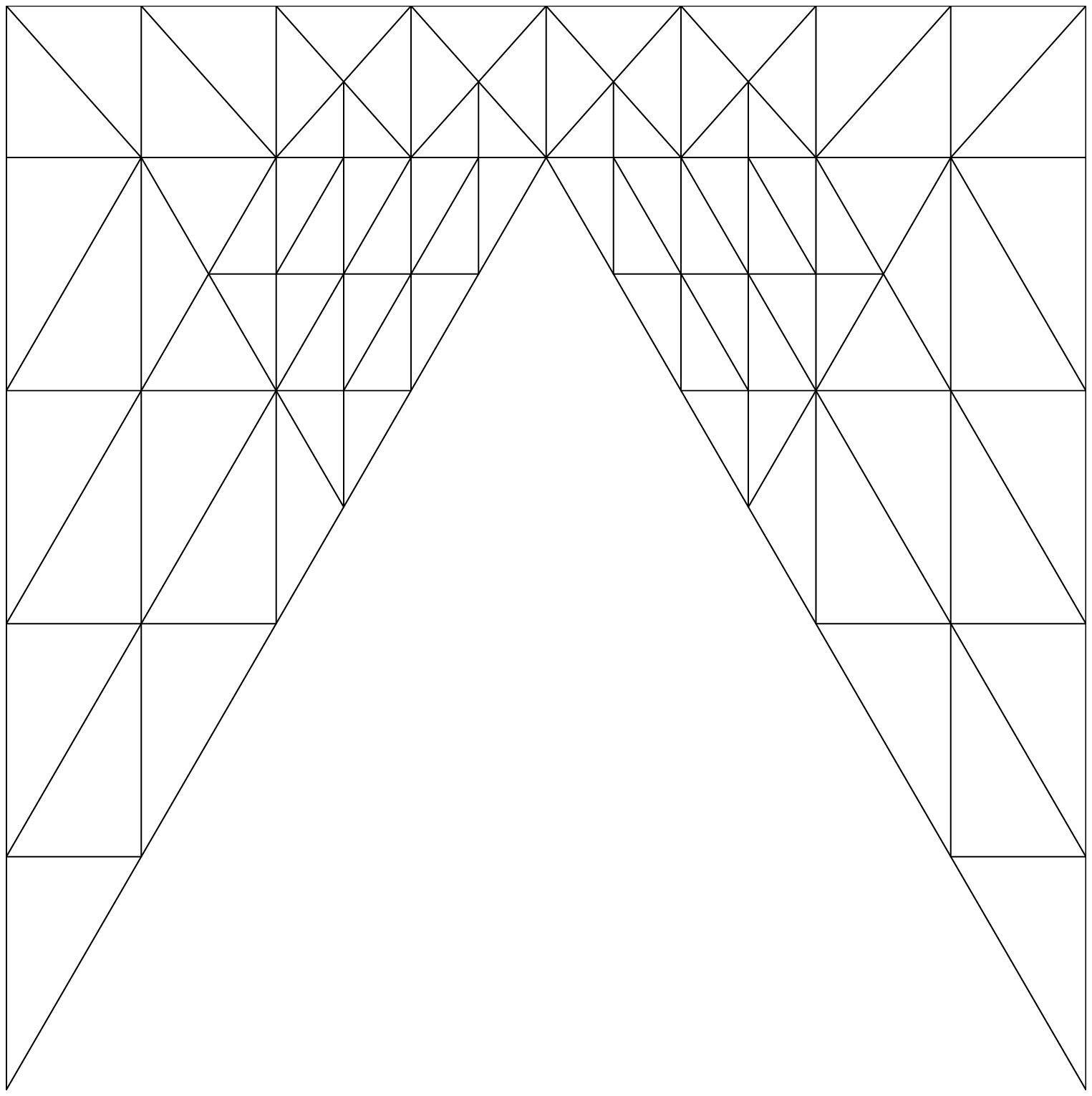}}\\\vspace{-0.3cm}
\subfigure[Step 4.]{\includegraphics[height=4.6cm, width=4.6cm]{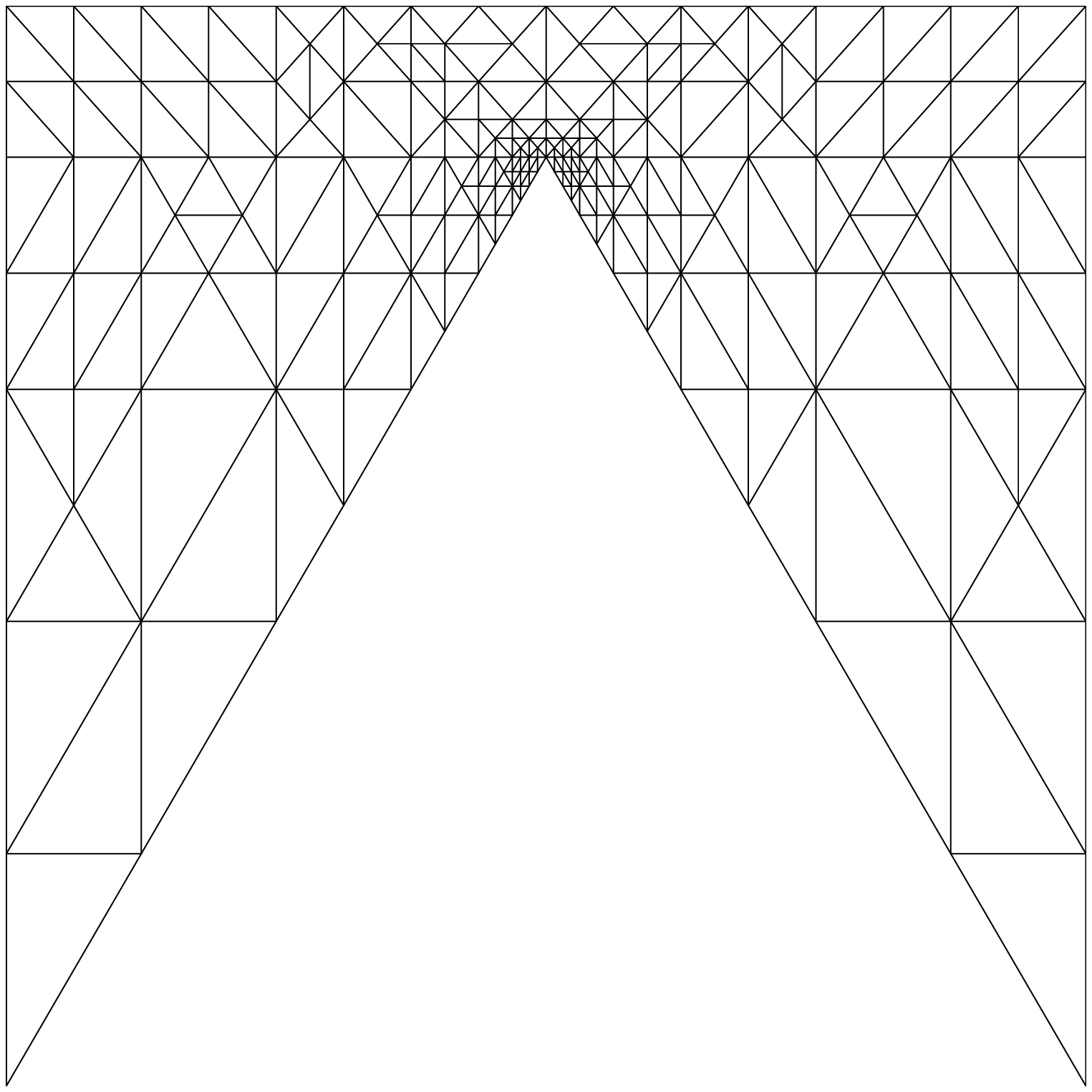}}
\subfigure[Step 6.]{\includegraphics[height=4.6cm, width=4.6cm]{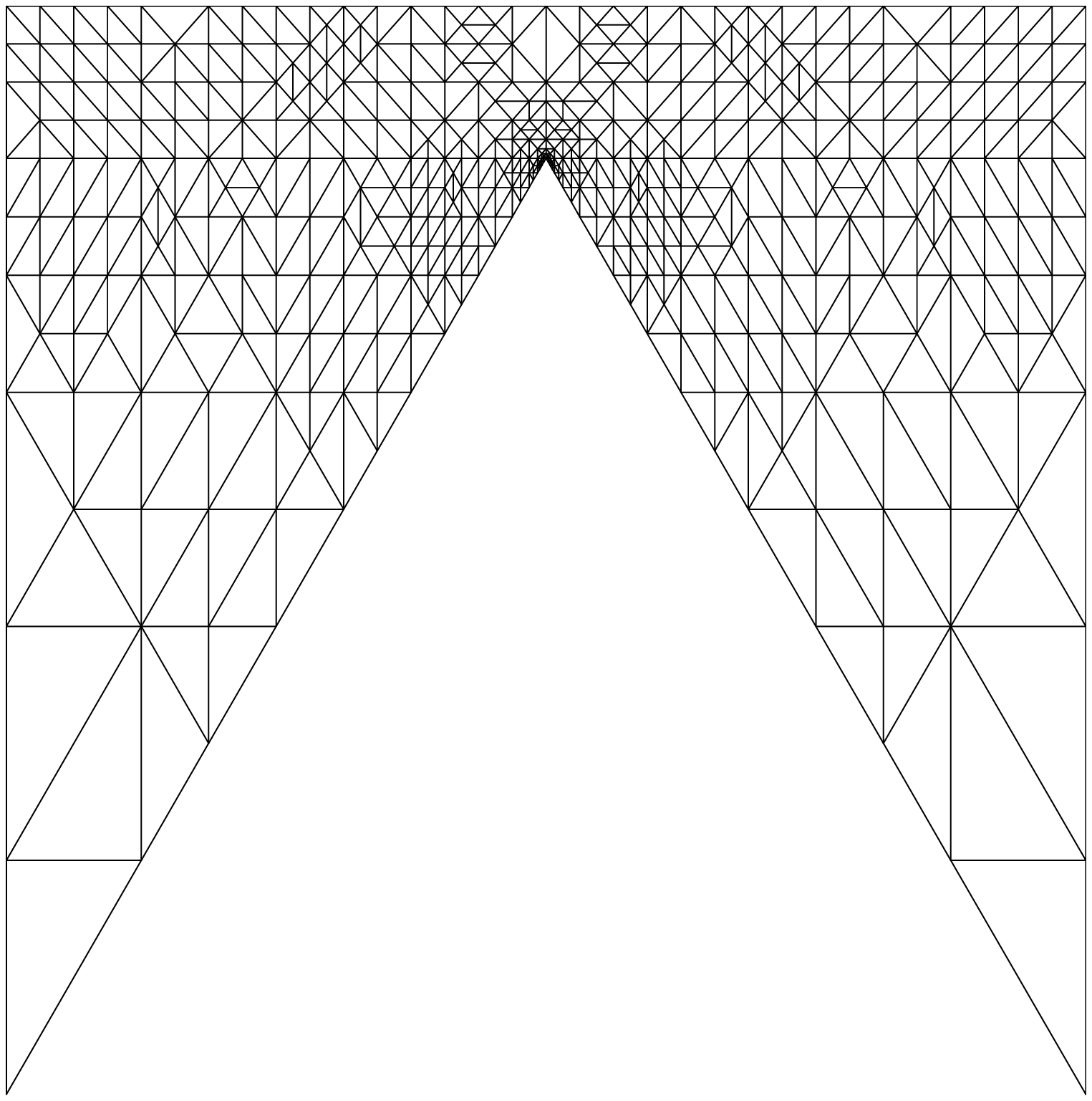}}
\vspace{-0.3cm} \caption{ Test 2. Adaptively refined meshes obtained with  FEM scheme at refinement steps 0, 1, 4 and 6.}
 \label{FIG:seis}
\end{figure}
In order to compute the errors $|\l_{1}-\l_{h1}|$,  due to the lack of an exact eigenvalue, we have used an approximation
based on a least squares fitting of the computed values obtained with
extremely refined meshes. Thus, we have obtained the value  $\l_{1}=1.9288$,
which has at least four correct significant digits.

We report in Table \ref{TABLA:5} the lowest
eigenvalue $\l_{h1}$ computed with each of the three schemes.
Each table includes the estimated convergence rate. 

\renewcommand{\arraystretch}{1.1}
\begin{table}[H]
\begin{center}
\caption{Test 2. Eigenvalue  $\l_{h1}$ computed with different schemes:  uniformly refined meshes (``Uniform FEM''), adaptively refined meshes with  FEM (``Adaptive FEM'') and adaptively refined meshes with VEM (``Adaptive VEM'').}
\begin{tabular}{|c|c||c|c||c|c|}
  \hline
    \multicolumn{2}{|c||}{Uniform FEM} &  \multicolumn{2}{|c||}{Adaptive VEM} & \multicolumn{2}{|c|}{Adaptive FEM} \\
    \hline
     $N$ & $\l_{h1}$  &   $N$ & $\l_{h1}$  &    $N$ & $\l_{h1}$  \\
\hline
  38  & 2.3083 &  38 &2.3083   &  38 & 2.3083\\ 
   123 &  2.0686 & 58   & 2.0721& 60 & 2.1067 \\  
   437  & 1.9828 & 106  & 1.9960& 85  &2.0362   \\
   1641   &1.9505 & 229   &1.9592  & 148&  1.9810   \\
   6353   &1.9377 & 350&   1.9467  & 185   &1.9678\\  
   14137   &1.9341& 666  &1.9384 & 280 &  1.9530  \\ 
   24993   &1.9325& 909 &  1.9354  & 458 &  1.9427  \\
   38291   &1.9316 & 1340 &  1.9329 & 646 &  1.9382   \\
   55921   &1.9310 &  2141  &1.9315 & 895   &1.9356   \\
   75993   &1.9306 &  3438&  1.9306  & 1593 &  1.9325   \\
   99137   &1.9303 & 5172   &1.9300  & 2122  & 1.9315  \\
   125353   &1.9301 & 8014 &  1.9296   & 3178 &  1.9306  \\
   154641   &1.9299 & 12365 &  1.9293 & 5341&  1.9298   \\
   187001   &1.9298  & 19153 &  1.9291   & 7522 &  1.9295  \\
   222433   &1.9297 &  29403 & 1.9290 & 11124 & 1.9292   \\
              \hline 
       Order  &$\mathcal{O}\left(N^{-0.68}\right)$ &      Order   & $\mathcal{O}\left(N^{-1.10}\right)$&      Order   &$\mathcal{O}\left(N^{-1.16}\right)$\\
       \hline
       $\l_1$  &1.9288 &      $\l_1$  &1.9288 &     $\l_1$  &1.9288\\
     \hline
\end{tabular}
\label{TABLA:5}
\end{center}
\end{table}

It can be seen from Table \ref{TABLA:5}, that the uniform refinement
leads to a convergence rate close to that  predicted by the
theory $\mathcal{O}\left(N^{-3/5}\right)$.
Instead, Tables~\ref{TABLA:5} show that the adaptive
VEM and FEM schemes allow us to recover the optimal  order of convergence $\mathcal{O}\left(N^{-1}\right)$.
This can be clearly seen from Figure~\ref{figurafinal}, where the three
error curves are reported. The plot also includes lines of slopes $-1$ and $-3/5$,
which correspond to the convergence rates of each scheme. 
 \begin{figure}[H]
\caption{Test 2. Error curves of  $|\l_1-\l_{h1}|$ for uniformly refined meshes (``Uniform FEM''), adaptively refined meshes with  FEM (``Adaptive FEM'') and adaptively refined meshes with VEM (``Adaptive VEM'').}
\centering\includegraphics[height=10cm, width=10cm]{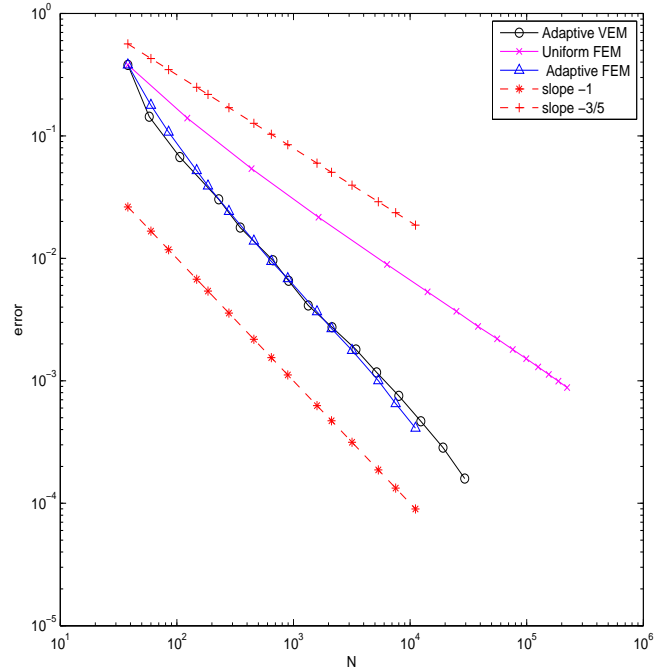}
\label{figurafinal}
\end{figure}

Finally, we report in Table~\ref{TABLA:N3} the same information
as in Table~\ref{TABLA:N2} for this test. Similar conclusions
as in the previous test follow from this table.
\begin{table}[H]
\begin{center}
\caption{Test 2. Components of the error estimator and effectivity indexes 
on the adaptively refined meshes with VEM.}
\begin{tabular}{|c|c|c|c|c|c|c|}
    \hline
     $N$ & $\l_{h1}$  &$|\l_{1}-\l_{h1}|$ &$\theta^{2}$ &$J^{2}$ & $\eta^{2}$ & $\dfrac{|\l_{1}-\l_{h1}|}{\eta^{2}}$\\
\hline
38 &2.3083 &   0.3795  &       0  &  2.3181  &  2.3181   & 0.1637\\
58   & 2.0721&    0.1433 &   0.0379 &   0.8231 &   0.8609 &   0.1664\\
 106  & 1.9960&    0.0672 &   0.0368 &   0.4188 &   0.4556 &   0.1475\\
 229   &1.9592  &  0.0304 &   0.0216  &  0.1942  &  0.2158  &  0.1408\\
 350&   1.9467  &  0.0179   & 0.0164  &  0.1359  &  0.1522  &  0.1173\\
 666  &1.9384 &   0.0096  &  0.0094  &  0.0749  &  0.0844  &  0.1143\\
 909 &  1.9354  &  0.0066   & 0.0068  &  0.0556  &  0.0624  &  0.1052\\
 1340 &  1.9329  &  0.0041 &   0.0047 &   0.0408 &   0.0454  &  0.0907\\
  2141  &1.9315 &   0.0027   & 0.0032 &   0.0275  &  0.0308  &  0.0891\\
  3438&  1.9306   & 0.0018&    0.0022 &   0.0178  &  0.0199  &  0.0904\\
               \hline 
      \end{tabular}
\label{TABLA:N3}
\end{center}
\end{table}

\section*{Conclusions}
We have derived an a posteriori error indicator for the VEM solution
of the Steklov eigenvalue problem. We have proved that it is efficient
and reliable. For lowest order elements
on triangular meshes, VEM coincides with FEM and the a posteriori error
indicators also coincide with the classical ones. However VEM allows
using general polygonal meshes including hanging nodes, which is particularly
interesting when designing an adaptive scheme. We have implemented such a
scheme driven by the proposed error indicators. We have assessed its
performance by means of a couple of tests which allow us to confirm
that the adaptive scheme yields optimal order of convergence for
regular as well as  singular solutions.  

\section*{Acknowledgments}

The authors warmfully thanks Louren\c{c}o Beir\~ao da Veiga from Universit\`a di Milano-Bicocca, Italy, by
many helpful discussions on this subject. The authors also thank Paola F. Antonietti from
Politecnico di Milano, Italy, by allowing us to use her code for  VEM adaptive mesh refinement. 

The first author was partially supported by CONICYT (Chile) through
FONDECYT project No. 1140791 and by DIUBB through project 151408 GI/VC,
Universidad del B\'io-B\'io, (Chile).
The second author was partially supported by a CONICYT (Chile)
fellowship.
The third author was partially supported by BASAL project, CMM,
Universidad de Chile (Chile).

\bibliographystyle{amsplain}

\end{document}